\font \sevenrm=cmr7
\font \fiverm=cmr5

\documentclass[11pt, english]{amsart}
\usepackage{amsfonts}
\usepackage{amssymb}
\usepackage{amsmath}
\usepackage{color}
\usepackage{graphicx}
\usepackage{xspace}
\usepackage{axodraw}

 \newcommand{\nc}{\newcommand}

\newenvironment{disarray}%
 {\everymath{\displaystyle\everymath{}}\array}%
 {\endarray}

\newtheorem{thm}{Theorem}

\newtheorem{cor}[thm]{Corollary}
\newtheorem{con}[thm]{Conjecture}
\newtheorem{lem}[thm]{Lemma}
\newtheorem{prop}[thm]{Proposition}
\newtheorem{defn}{Definition}
\newtheorem{rmk}[thm]{Remark}

\nc{\comment}[1]{[[{\tt #1}]] }
\nc{\Cal}[1]{{\mathcal {#1}}}
\nc{\mop}[1]{\mathop{\hbox {\rm #1} }\nolimits}
\nc{\gmop}[1]{\mathop{\hbox {\bf #1} }\nolimits}

\nc{\smop}[1]{\mathop{\hbox {\sevenrm #1} }\nolimits}
\nc{\ssmop}[1]{\mathop{\hbox {\fiverm #1} }\nolimits}
\nc{\mopl}[1]{\mathop{\hbox {\rm #1} }\limits}
\nc{\smopl}[1]{\mathop{\hbox {\sevenrm #1} }\limits}
\nc{\ssmopl}[1]{\mathop{\hbox {\fiverm #1} }\limits}
\nc{\frakg}{{\frak g}}
\nc{\g}[1]{{\frak {#1}}}
\def \restr#1{\mathstrut_{\textstyle |}\raise-6pt\hbox{$\scriptstyle #1$}}
\def \srestr#1{\mathstrut_{\scriptstyle |}\hbox to
  -1.5pt{}\raise-4pt\hbox{$\scriptscriptstyle #1$}}
\nc{\wt}{\widetilde} \nc{\wh}{\widehat}

\nc{\redtext}[1]{\textcolor{red}{#1}}
\nc{\bluetext}[1]{\textcolor{blue}{#1}}

\nc\fleche[1]{\mathop{\hbox to #1 mm{\rightarrowfill}}\limits}
\nc{\ignore}[1]{}
\def\semi{\mathrel{\times}\kern -.85pt\joinrel\mathrel{\raise
    1.4pt\hbox{${\scriptscriptstyle |}$}}}
\nc\R{{\mathbb R}}
\nc\N{{\mathbb N}}
\nc\inver{^{-1}}
\nc\point{\hbox{\bf .}}
\nc\un{\hbox{\bf 1}}

\def\racine{{\scalebox{0.3}{ 
\begin{picture}(12,12)(38,-38)
\SetWidth{0.5} \SetColor{Black} \Vertex(45,-33){5.66}
\end{picture}}}}
  
 \def\arbrea{\,{\scalebox{0.15}{ 
  \begin{picture}(8,55) (370,-248)
    \SetWidth{2}
    \SetColor{Black}
    \Line(374,-244)(374,-200)
    \Vertex(374,-197){9}
    \Vertex(375,-245){12}
  \end{picture}
}}\,}

 \def\arbreba{\,{\scalebox{0.15}{ 
\begin{picture}(8,106) (370,-197)
    \SetWidth{2}
    \SetColor{Black}
    \Line(374,-193)(374,-149)
    \Vertex(374,-146){9}
    \Vertex(375,-194){12}
    \Line(374,-142)(374,-98)
    \Vertex(374,-95){9}
  \end{picture}
}}\,}

 \def\arbrebb{\,{\scalebox{0.15}{ 
  \begin{picture}(48,48) (349,-255)
    \SetWidth{2}
    \SetColor{Black}
    \Vertex(375,-252){12}
    \Line(376,-250)(395,-215)
    \Line(373,-251)(354,-214)
    \Vertex(353,-211){9}
    \Vertex(395,-213){9}
  \end{picture}
}}}

\def\arbreca{\,{\scalebox{0.15}{
\begin{picture}(8,156) (370,-147)
    \SetWidth{2}
    \SetColor{Black}
    \Line(374,-143)(374,-99)
    \Vertex(374,-96){9}
    \Vertex(375,-144){12}
    \Line(374,-92)(374,-48)
    \Vertex(374,-45){9}
    \Line(374,-42)(374,2)
    \Vertex(374,5){9}
  \end{picture}
}}\,}

\def\arbrecb{\,{\scalebox{0.15}{
\begin{picture}(48,94) (349,-255)
\SetWidth{2}
\SetColor{Black}
\Line(376,-204)(395,-169)
\Line(373,-205)(354,-168)
\Vertex(353,-165){9}
\Vertex(395,-167){9}
\Vertex(374,-205){9}
\Line(374,-246)(374,-209)
\Vertex(374,-252){12}
\end{picture}}}\,}

\def\arbrecc{\,{\scalebox{0.15}{
 \begin{picture}(48,98) (349,-205)
    \SetWidth{2}
    \SetColor{Black}
    \Vertex(375,-202){12}
    \Line(376,-200)(395,-165)
    \Line(373,-201)(354,-164)
    \Vertex(353,-161){9}
    \Vertex(395,-163){9}
    \Line(353,-160)(353,-113)
    \Vertex(353,-111){9}
  \end{picture}
}}\,}

\def\arbrecd{\,{\scalebox{0.15}{
\begin{picture}(48,52) (349,-251)
    \SetWidth{2}
    \SetColor{Black}
    \Vertex(375,-248){12}
    \Line(376,-246)(395,-211)
    \Line(373,-247)(354,-210)
    \Vertex(353,-207){9}
    \Vertex(395,-209){9}
    \Line(375,-247)(375,-206)
    \Vertex(376,-203){9}
  \end{picture}
 }}\,}

\def\arbreda{\,{\scalebox{0.15}{
\begin{picture}(8,204) (370,-99)
    \SetWidth{2}
    \SetColor{Black}
    \Line(374,-95)(374,-51)
    \Vertex(374,-48){9}
    \Vertex(375,-96){12}
    \Line(374,-44)(374,0)
    \Vertex(374,3){9}
    \Line(374,6)(374,50)
    \Vertex(374,53){9}
    \Line(374,53)(374,98)
    \Vertex(374,101){9}
  \end{picture}
}}\,}

\def\arbredb{\,{\scalebox{0.15}{
\begin{picture}(48,135) (349,-255)
    \SetWidth{2}
    \SetColor{Black}
    \Line(376,-163)(395,-128)
    \Line(373,-164)(354,-127)
    \Vertex(353,-124){9}
    \Vertex(395,-126){9}
    \Vertex(374,-164){9}
    \Line(374,-205)(374,-168)
    \Vertex(374,-207){9}
    \Line(374,-248)(374,-211)
    \Vertex(374,-252){12}
  \end{picture}
}}\,}

\def\arbredc{\,{\scalebox{0.15}{
 \begin{picture}(48,150) (349,-205)
    \SetWidth{2}
    \SetColor{Black}
    \Line(376,-148)(395,-113)
    \Line(373,-149)(354,-112)
    \Vertex(353,-109){9}
    \Vertex(395,-111){9}
    \Line(353,-108)(353,-61)
    \Vertex(353,-59){9}
    \Line(374,-200)(374,-153)
    \Vertex(374,-149){9}
    \Vertex(374,-202){12}
  \end{picture}
}}\,}

\def\arbredd{\,{\scalebox{0.15}{
 \begin{picture}(48,99) (349,-251)
    \SetWidth{2}
    \SetColor{Black}
    \Line(376,-199)(395,-164)
    \Line(373,-200)(354,-163)
    \Vertex(353,-160){9}
    \Vertex(395,-162){9}
    \Vertex(376,-156){9}
    \Vertex(376,-248){12}
    \Line(375,-245)(375,-204)
    \Line(375,-200)(375,-159)
    \Vertex(375,-201){9}
  \end{picture}
}}\,}

\def\arbrede{\,{\scalebox{0.15}{
 \begin{picture}(48,153) (349,-150)
    \SetWidth{2}
    \SetColor{Black}
    \Vertex(375,-147){12}
    \Line(376,-145)(395,-110)
    \Line(373,-146)(354,-109)
    \Vertex(353,-106){9}
    \Vertex(395,-108){9}
    \Line(353,-105)(353,-58)
    \Vertex(353,-56){9}
    \Line(353,-52)(353,-5)
    \Vertex(353,-1){9}
  \end{picture}
}}\,}

\def\arbredf{\,{\scalebox{0.15}{
\begin{picture}(48,98) (349,-205)
    \SetWidth{2}
    \SetColor{Black}
    \Vertex(375,-202){12}
    \Line(376,-200)(395,-165)
    \Line(373,-201)(354,-164)
    \Vertex(353,-161){9}
    \Vertex(395,-163){9}
    \Line(353,-160)(353,-113)
    \Vertex(353,-111){9}
    \Line(395,-159)(395,-112)
    \Vertex(395,-111){9}
  \end{picture}
}}\,}

\def\arbredz{\,{\scalebox{0.15}{
  \begin{picture}(68,88) (329,-215)
    \SetWidth{2}
    \SetColor{Black}
    \Vertex(375,-212){12}
    \Line(376,-210)(395,-175)
    \Line(373,-211)(354,-174)
    \Vertex(353,-171){9}
    \Vertex(395,-173){9}
    \Line(351,-168)(332,-131)
    \Line(355,-168)(374,-133)
    \Vertex(333,-131){9}
    \Vertex(374,-131){9}
  \end{picture}
}}\,}

\def\arbredg{\,{\scalebox{0.15}{
\begin{picture}(48,98) (349,-205)
    \SetWidth{2}
    \SetColor{Black}
    \Vertex(375,-202){12}
    \Line(376,-200)(395,-165)
    \Line(373,-201)(354,-164)
    \Vertex(353,-161){9}
    \Vertex(395,-163){9}
    \Line(375,-201)(375,-160)
    \Vertex(376,-157){9}
    \Vertex(376,-111){9}
    \Line(375,-155)(375,-114)
  \end{picture}
}}\,}

\def\arbredh{\,{\scalebox{0.15}{
 \begin{picture}(90,46) (330,-257)
    \SetWidth{2}
    \SetColor{Black}
    \Vertex(375,-254){12}
    \Line(376,-252)(395,-217)
    \Vertex(395,-215){9}
    \Line(374,-254)(335,-226)
    \Vertex(334,-224){9}
    \Line(375,-252)(356,-215)
    \Vertex(355,-215){9}
    \Line(374,-255)(417,-227)
    \Vertex(418,-225){9}
  \end{picture}
}}\,}

\def\arbreea{\,{\scalebox{0.15}{
 \begin{picture}(8,251) (370,-52)
    \SetWidth{2}
    \SetColor{Black}
    \Line(374,-48)(374,-4)
    \Vertex(374,-1){9}
    \Vertex(375,-49){12}
    \Line(374,3)(374,47)
    \Vertex(374,50){9}
    \Line(374,53)(374,97)
    \Vertex(374,100){9}
    \Vertex(374,148){9}
    \Line(374,149)(374,194)
    \Line(374,100)(374,144)
    \Vertex(374,195){9}
  \end{picture}
}}\,}

\def\arbreez{\,{\scalebox{0.15}{
\begin{picture}(90,48) (330,-255)
    \SetWidth{2}
    \SetColor{Black}
    \Vertex(375,-252){12}
    \Line(376,-250)(395,-215)
    \Vertex(395,-213){9}
    \Line(374,-252)(335,-224)
    \Vertex(334,-222){9}
    \Line(375,-250)(356,-213)
    \Vertex(355,-213){9}
    \Line(374,-253)(417,-225)
    \Vertex(418,-223){9}
    \Line(375,-252)(375,-210)
    \Vertex(375,-211){9}
  \end{picture}}}\,}
  
\def\shu{\joinrel{\!\scriptstyle\amalg\hskip -3.1pt\amalg}\,}
  
\begin{document}

\title[Two interacting Hopf algebras of trees]{Two interacting Hopf algebras of trees\\[0.2cm] {\small{a Hopf-algebraic approach to composition\\ and
substitution of B-series}}}


\author{Damien Calaque}
\address{Dept. of Math., 
	   ETH Z\"urich, 8092 Z\"urich, Switzerland. On leave of absence from Univ. Lyon 1, France.}
	   \email{calaque@math.univ-lyon1.fr,damien.calaque@math.ethz.ch}
 	  \urladdr{http://math.univ-lyon1.fr/~calaque/}

\author{Kurusch Ebrahimi-Fard}
\address{Universit\'e de Haute Alsace,
         4, rue des Fr\`eres Lumi\`ere,
         68093 Mulhouse, France}
         \email{kurusch.ebrahimi-fard@uha.fr}         
	  \urladdr{http://www.th.physik.uni-bonn.de/th/People/fard/}

\author{Dominique Manchon}
\address{Universit\'e Blaise Pascal,
         C.N.R.S.-UMR 6620,
         63177 Aubi\`ere, France}       
         \email{manchon@math.univ-bpclermont.fr}
         \urladdr{http://math.univ-bpclermont.fr/~manchon/}

\date{August, 24th 2009}

\begin{abstract}
Hopf algebra structures on rooted trees are by now a well-studied object, especially in the context of combinatorics. In this work we consider a Hopf algebra $\Cal H$ by introducing a coproduct on a (commutative) algebra of rooted forests, considering each tree of the forest (which must contain at least one edge) as a Feynman-like graph without loops. The primitive part of the graded dual is endowed with a pre-Lie product defined in terms of insertion of a tree inside another. We establish a surprising link between the Hopf algebra $\Cal H$ obtained this way and the well-known Connes--Kreimer Hopf algebra of rooted trees $\Cal H_{\makebox{{\tiny{\rm{CK}}}}}$ by means of a natural $\Cal H$-bicomodule structure on $\Cal H_{\makebox{{\tiny{CK}}}}$. This enables us to recover recent results in the field of numerical methods for differential equations due to Chartier, Hairer and Vilmart as well as Murua. 

%
%

\bigskip


\smallskip


\end{abstract}

%
%

\maketitle





\section{Introduction}
\label{sect:intro}

Since the pioneering work of Cayley \cite{Cay} on the use of rooted trees in the context of differential equations, many more instances where trees play a significant role have appeared in the literature. The work of  Butcher \cite{Butcher1}, Grossman and Larson \cite{GL} and, more recently, Munthe--Kaas and Wright~\cite{MKW} in the field of numerical analysis, as well as the seminal findings of Kreimer and Connes in the context of the process of renormalization in perturbative quantum field theory \cite{CK1} marked such moments. In these works it is the notion of Hopf algebra defined on rooted trees that characterizes genuine combinatorial aspects of the underlying problems. The discovery of these algebraic structures lead to more transparency eventually allowing to obtain profound insights which significantly altered the picture. (See, for example, \cite{CM,CK2}.) Since then the field of combinatorial Hopf algebra has emerged and many more rooted tree Hopf algebras have been studied. 

\smallskip

Recall Rota and Joni's \cite{JR} observation that various combinatorial objects naturally possess compatible product and coproduct structures, ultimately converging into the notion of graded connected Hopf algebras, now referred to as combinatorial Hopf algebras. See Schmitt's seminal paper \cite{Sch}. Generally speaking, such combinatorial Hopf algebras consist of a graded vector space where the homogeneous components are spanned by finite sets of combinatorial objects, such as planar or non-planar rooted trees, and the (co)algebraic structures are given by particular constructions on those objects. In fact, roughly, one may distinguish between two complementary classes of combinatorial Hopf algebras on rooted trees, those with a simple algebra structure and a more complicated coalgebra structure and their graded duals. The aforementioned works by Grossman and Larson, Munthe--Kaas and Wright, and Kreimer and Connes provide genuine examples. (See also Loday and Ronco \cite{LR}.) 

In this paper, we introduce on the graded commutative polynomial algebra of finite type over the field $k$ generated by rooted forests a coproduct defined by considering each tree of the forest as a Feynman graph without loops. This Hopf algebra with its coproduct, is graded by the number of edges denoted by $e$. It differs from Connes--Kreimer's rooted tree Hopf algebra \cite{CK1}, see also \cite{K}, which is graded by the number $v$ of vertices, and where the coproduct is defined in terms of admissible cuts. For each rooted tree $t$ one has $v(t) = e(t) + 1$. We show that the primitive part of the graded dual is endowed with a natural pre-Lie product defined in terms of insertion of a tree inside another.

However, we establish an apparently surprising and useful link between the Hopf algebra $\Cal H$ obtained this way and the aforementioned Connes--Kreimer Hopf algebra of rooted trees $\Cal H_{\makebox{{\tiny{CK}}}}$ by means of a $\Cal H$-bicomodule structure on $\Cal H_{\makebox{{\tiny{CK}}}}$ obtained by a natural slight extension of the coproduct of $\Cal H$. More precisely we show (see Proposition \ref{biderivation}) that for any infinitesimal character $a$ of $\Cal H$, the associated left coaction operator $^{\rm{t}}\! L_a: \Cal H_{\makebox{{\tiny{CK}}}} \to \Cal H_{\makebox{{\tiny{CK}}}}$ is a biderivation. This in turn defines a Hopf algebra automorphism $^{\rm{t}}\! L_\varphi$ for any character $\varphi$ of the Hopf algebra $\Cal H$.

We recover this way results due to Chartier, Hairer and Vilmart \cite{CHV,CHVj,CHV09} about the compatibility between two products, which, roughly speaking, can be considered as the convolution products with respect to the two Hopf algebras above, see also \cite{CHV09,Vil08}. We compare the backward error analysis character $\omega$ of \cite{CHV, M} with the element $\log^*$ of Chapoton \cite{C}, thus applying the recursive formulas of \cite{CHV, M} to the pre-Lie Magnus expansion of \cite{EM}.  An operadic interpretation is outlined, and we finally recover some nice properties of the application $\omega$ by mapping surjectively Connes--Kreimer's Hopf algebra of rooted trees onto Hoffman's quasi-shuffle Hopf algebra in one generator~\cite{H1}.

\smallskip

The paper is organized a follows. Section \ref{sect:prelim} and section \ref{sect:rt}  briefly recall the basic notions of Hopf algebra and rooted trees, respectively. In section \ref{sect:hopf} a particular commutative and non-cocommutative Hopf algebra $\Cal H$ of rooted trees is studied in detail. The following two sections elaborate on the coproduct of this Hopf algebra and present refined combinatorial ways to calculate it. In section \ref{sect:antipo} a non-recursive formula for the antipode of $\Cal H$ is given. In section \ref{sect:log} a particular Hopf algebra character and its inverse are introduced, which are naturally related to the notion of backward error analysis. The following section \ref{sect:ckhopf} briefly recalls the Connes--Kreimer Hopf algebra $\Cal H_{\makebox{{\tiny{CK}}}}$ of rooted trees and establishes a relation to the rooted tree Hopf algebra $\Cal H$ by introducing a $\Cal H$-bicomodule structure on $\Cal H_{\makebox{{\tiny{CK}}}}$. Section \ref{sect:B-series} recalls the composition of $B$-series. The Hopf algebraic interpretation of Chartier--Hairer--Vilmart's substitution law for $B$-series is given. Section \ref{sect:comp} is a continuation of section \ref{sect:log}. Here a link to the quasi-shuffle Hopf algebra is established and used to calculate a particular infinitesimal character on the Connes--Kreimer Hopf algebra $\Cal H_{\makebox{{\tiny{CK}}}}$. Finally, the last section contains some explicit calculation of the antipode of $\Cal H$.


\section{Preliminaries}
\label{sect:prelim}

For later use in this subsection we briefly outline the (co)algebraic setting
in which we will work. Here, $k$ denotes the ground field (which will be
supposed to be of characteristic zero) over which all algebraic structures are
defined. The term {\it{algebra}} always means unital associative $k$-algebra,
denoted by the triple $(\Cal A,m_{\Cal A},\eta_{\Cal A})$, where $\Cal A$ is a
$k$-vector space with a product $m_{\Cal A}: \Cal A \otimes \Cal A \to \Cal A$
and a unit map $\eta_{\Cal A}: k \to \Cal A$. {\it{Coalgebras}} over $k$ are
denoted by the triple $(\Cal C,\Delta_{\Cal C},\epsilon_{\Cal C})$, where the
coproduct map $\Delta_{\Cal C}: \Cal C \to \Cal C \otimes \Cal C$ is
coassociative and $\epsilon_{\Cal C}: \Cal C \to k$ denotes the counit map. A
subspace $\Cal J \subset \Cal C$ is called a left (right) {\it{coideal}} if
$\Delta_{\Cal C}(\Cal J) \subset \Cal C \otimes \Cal J$ ($\Delta_{\Cal C}(\Cal
J) \subset \Cal J \otimes \Cal C$). A right {\it{comodule}} over $\Cal
C$ is a $k$-vector space $\Cal M$ together with a linear map $\psi: \Cal M \to
\Cal M \otimes \Cal C$, such that $(\mop{Id}_{\Cal M } \otimes \Delta_{\Cal
  C})\circ \psi = (\psi \otimes \mop{Id}_{ \Cal C}) \circ  \psi$ and
$(\mop{Id}_{\Cal M } \otimes \epsilon_{\Cal C})\circ \psi = \mop{Id}_{\Cal M
}$ (analogously for left comodules). A {\it{bicomodule}} over $\Cal
C$ is a $k$-vector space $\Cal M$ together with two linear maps $\psi: \Cal M
\to\Cal M \otimes \Cal C$ and $\phi: \Cal M\to  \Cal C\otimes\Cal M$ which
endow $\Cal M$ with a structure of right- and left comodule, respectively, and
such that the compatibility condition:
\begin{equation}
(\mop{Id}_{\Cal C}\otimes\psi)\circ\phi=(\phi\otimes\mop{Id}_{\Cal C})\circ\psi
\end{equation}
holds. A {\it{bialgebra}} consists of an algebra and coalgebra structure together with compatibility relations. We denote a {\it{Hopf algebra}} by $(\Cal H,m_{\Cal H},\eta_{\Cal H},\Delta_{\Cal H},\epsilon_{\Cal H},S)$. It is a bialgebra together with a particular $k$-linear map, i.e. the {\it{antipode}} $S: \Cal H \to \Cal H$, satisfying the Hopf algebra axioms~\cite{Abe80,Berg85,Sw69}.  In the following we omit subscripts if there is no danger of confusion.  

Let $\Cal H$ be a connected filtered bialgebra:
$$
    k=\Cal H^{(0)} \subset \Cal H^{(1)} \subset \cdots \subset 
      \Cal H^{(n)} \subset \cdots, \hskip 6mm \bigcup_{n\ge 0} \Cal H^{(n)}=\Cal H,
$$
and let $\Cal A$ be any commutative algebra. The space $\Cal
L(\Cal H,\Cal A)$ of linear maps from $\Cal H$ to $\Cal A$
together with the convolution product $f \star g := m_{\Cal A}
\circ (f \otimes g) \circ \Delta$, $ f,g \in \Cal L(\Cal H,\Cal A)$ is an algebra with unit $e:=\eta_{\Cal A} \circ \epsilon$.
The filtration of $\Cal H$ implies a decreasing filtration on
$\Cal L(\Cal H,\Cal A)$ and $\Cal L(\Cal H,\Cal
A)$ is complete with respect to the induced topology, e.g. see \cite{Man} for more details.
The subset $\g g_0 \subset \Cal L(\Cal H,\Cal A)$ of
linear maps $\alpha$ that send the bialgebra unit to zero,
$\alpha(1)=0$, forms a Lie algebra in $\Cal L(\Cal H,\Cal A)$. The
exponential:
$$
    \exp^\star(\alpha) = \sum_{k\ge 0} \frac{1}{k!}\alpha^{\star k}
$$
makes sense and is a bijection from $\g g_0$ onto the group
$G_0=e+\g g_0$ of linear maps $\gamma$ that send the bialgebra
unit to the algebra unit, $\alpha(1)=1_{\Cal A}$.

An {\it{infinitesimal character}} with values in $\Cal A$ is a linear map
$\xi \in \Cal L(\Cal H,\Cal A)$ such that for $x,y \in \Cal H$:
\begin{equation}
\label{infinitesimal}
    \xi(xy) = \xi(x)e(y) + e(x)\xi(y).
\end{equation}
We denote by $\g g_{\Cal A} \subset \g g_0$ the linear space of
infinitesimal characters. We call an $\Cal A$-valued map $\rho$ in
$\Cal L(\Cal H,\Cal A)$ a {\it{character}} if for $x,y \in \Cal H$:
\begin{equation}
\label{character}
    \rho(xy) = \rho(x)\rho(y).
\end{equation}
The set of such unital algebra morphisms is denoted by $G_{\Cal A} \subset G_0$. It is easily verified that the set $G_{\Cal A}$ of characters from $\Cal H$ to $\Cal A$ forms a group for the convolution product. In fact it is the pro-unipotent {\sl{group of $\Cal A$-valued morphisms on the bialgebra $\Cal H$}\/}{}. And $\g g_{\Cal A}$ in $\g g_0$ is the corresponding pro-nilpotent Lie algebra. The exponential map $\exp^{\star}$ restricts to a bijection between $\g g_{\Cal A}$ and $G_{\Cal A}$. The neutral element $e:=\eta_{\Cal A}\circ \epsilon$  in $G_{\Cal A}$ is given
by $e(1)=1_{\Cal A}$ and $e(x)=0$ for $x \in \mop{Ker}\epsilon$. The inverse of $\varphi \in G_{\Cal A}$ is given by composition with the antipode $S$:
\begin{equation}
\label{inverse}
    \varphi^{\star -1} = \varphi \circ S.
\end{equation}
Recall the twisting map $\tau_{\Cal V} (a \otimes b):=b \otimes a$ for $a,b$, say, in the  $k$-vector space $\Cal V$.

For completeness we ask the reader to recall the notion of left {\it{pre-Lie algebra}} \cite{AG,ChaLiv,DL}. A left pre-Lie algebra is a $k$-vector space $\Cal P$ with a binary composition $\rhd$ that satisfies the left pre-Lie identity:
\begin{eqnarray}
\label{prelie}
    (a\rhd b)\rhd c-a\rhd(b\rhd c)&=& (b\rhd a)\rhd c-b\rhd(a\rhd c),
\end{eqnarray}
for $a,b,c \in \Cal P$ and with an analogous identity for right pre-Lie algebra. For $a,b \in \Cal P$ the bracket $[a,b]:=a \rhd b - b \rhd a$ satisfies the Jacobi identity and therefore turns $\Cal P$ into a Lie algebra.


\section{Rooted trees and rooted forests}
\label{sect:rt}

Recall, that a --non-planar-- {\sl rooted tree\/} is either the empty set, or a finite connected simply connected oriented graph such that every vertex has exactly one incoming edge, except for a distinguished vertex (the root) which has no incoming egde. A vertex without outgoing edges is called a leaf. 
We list all rooted trees up to degree $5$: 
$$
	\racine,\hskip 8mm  \arbrea,\hskip 8mm  \arbreba,\  \arbrebb,\hskip 8mm  \arbreca,\ \arbrecb,\ \arbrecc,\ \arbrecd,\hskip 8mm 
	 \arbreda,\ \arbredb,\ \arbredc,\ \arbredd,\ \arbrede,\ \arbredf,\
     \arbredz,\ \arbredg,\ \arbredh.
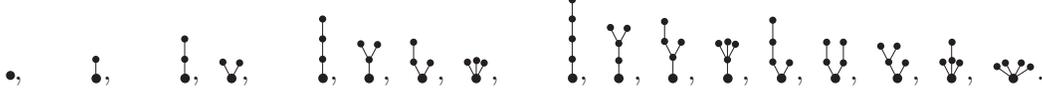
$$ 
A {\sl rooted forest\/} is a finite collection $s=(t_1,\ldots, t_n)$ of rooted trees, which we simply denote by the (commutative) product $t_1 \cdots t_n$. The operator $B_+$ (e.g. see \cite{F}) associates to the forest $s$ the tree $B_+(s)$ obtained by grafting the root of each connected component $t_j$ on a common new root. $B_+(\emptyset)$ is the unique rooted tree $\begin{matrix}  \\[-0.5cm]  \racine \end{matrix}$ with only one vertex.  Recall the definition of several important numbers associated to a rooted forest $s = t_1 \cdots t_n$:
\begin{enumerate}

\item The number of vertices $v(s)=\sum_{j=1}^{n} v(t_j)$.

\item The number of edges $e(s)=\sum_{j=1}^{n} e(t_j)$. For a non-empty tree $t$ we have $e(t)=v(t)-1$.

\item The tree factorial, recursively defined (with respect to the number of vertices) by $\begin{matrix}  \\[-0.5cm]  \racine \end{matrix}\hspace{-0.1cm}\ !=1$ and:
\begin{equation*}
	\big(B_+(t_1 \cdots t_n)\big)! = v\big(B_+(t_1\cdots t_n)\big)\prod_{j=1}^nt_j!
		=\Big(1+\sum_{j=1}^nv(t_j)\Big)\prod_{j=1}^nt_j!.
\end{equation*}
The tree factorial is multiplicatively extended to forests: $(t_1\cdots t_n)! := t_1! \cdots t_n!$.

\item The internal symmetry factor
  $\sigma(s)=\prod_{j=1}^n|\mop{Aut}t_j|$. Note that it differs
  from the true symmetry factor $\sigma\big(B_+(s)\big)$ of the forest $s$,
  which is not multiplicative.
  
\item The Connes--Moscovici coefficient of a tree $t$, e.g. see \cite{Br, F,KreimerChen}, defined by:
\begin{equation}
\label{CM}
	\mop{CM}(t)=\frac{v(t)!}{t!\sigma(t)}.
\end{equation}
\end{enumerate}


\section{The rooted tree Hopf algebra $\Cal H$}
\label{sect:hopf}


\subsection{Definition}
\label{def}

Let $T$ be the $k$-vector space spanned by -non-planar- rooted trees, {\sl excluding the
  empty tree\/}. Write $T=T' \oplus k \begin{matrix}  \\[-0.5cm]  \racine
\end{matrix}$ where $\begin{matrix}  \\[-0.5cm]  \racine \end{matrix}$ stands
for the unique one-vertex tree, and where $T'$ is the $k$-vector space spanned by
rooted trees with at least one edge. Consider the symmetric algebra $\Cal H =
\Cal S(T')$, which can be seen as the $k$-vector space generated by rooted
forests with all connected components containing at least one edge. One identifies the unit of $ \Cal S(T')$ with the rooted tree $\begin{matrix}  \\[-0.5cm]  \racine \end{matrix}$~. A {\sl subforest\/} of a tree $t$ is either the trivial forest $\begin{matrix}  \\[-0.5cm]  \racine \end{matrix}\hspace{0.1cm}\! $, or a collection $(t_1,\ldots ,t_n)$ of pairwise disjoint subtrees of $t$, each of them containing at least one edge. In particular two subtrees of a subforest cannot have any common vertex.\\

Let $s$ be a subforest of a rooted tree $t$. Denote by $t / s$ the tree obtained by contracting each connected component of $s$ onto a vertex. We turn $\Cal H $ into a bialgebra by defining a coproduct $\Delta : \Cal H \to \Cal H \otimes  \Cal H$ on each tree $t \in T'$ by~:
\begin{equation}
\label{newcoprod}
	\Delta(t)=\sum_{s\subseteq t} s \otimes t/s,
\end{equation}
where the sum runs over all possible subforests (including the unit $\begin{matrix}  \\[-0.5cm]  \racine \end{matrix}$ and the full subforest $t$). Here are two examples:
\allowdisplaybreaks{
\begin{eqnarray*}
	\Delta(\arbreba) &=&\arbreba \otimes \racine + \racine \otimes \arbreba + 2 \arbrea \otimes \arbrea\\
	\Delta(\arbrecd) &=&\arbrecd \otimes \racine + \racine \otimes \arbrecd + 3 \arbrebb \otimes \arbrea + 3  \arbrea \otimes \arbrebb\ .
\end{eqnarray*}}
As usual we extend the coproduct $\Delta$ multiplicatively onto $ \Cal S(T')$. In fact, co-associativity is easily verified. This makes $\Cal H:=\bigoplus_{n \ge 0}\Cal H_n$ a connected graded bialgebra, hence a Hopf algebra, where the grading is defined in terms of the number of edges. The antipode $S : \Cal H \to \Cal H$ is given (recursively with respect to the number of edges) by one of the two following formulas:
\begin{eqnarray}
	S(t) & = & -t-\sum_{s, \racine \, \not= s\subsetneq t}S(s)\ t/s, \\
	S(t) & = & -t-\sum_{s, \racine \, \not= s\subsetneq t}s\ S(t/s).
	\label{eq-anti-recu}
\end{eqnarray}

\subsection{The rooted tree bialgebra $\wt {\Cal H}$}
\label{ssect:bialgebra}

Consider the symmetric algebra $\wt {\Cal H}=S(T)$, which can be seen as the $k$-vector space generated by rooted forests. A {\sl spanning subforest\/} of a tree $t$ is a collection $(t_1,\ldots ,t_n)$ of disjoint subtrees, such that any vertex of $t$ belongs to one (unique) $t_j$. The contracted tree $t/s$ is defined the same way as above, except that the contraction is effective only for subtrees with at least two vertices. This is not a Hopf algebra since there is no inverse for the grouplike element $\begin{matrix}  \\[-0.5cm]  \racine \end{matrix}\hspace{0.1cm}\! $, but the Hopf algebra $\Cal H$ can be recovered as $\wt{\Cal H}/\Cal J$, where $\Cal J$ is the ideal of $\wt{\Cal H}$ spanned by $\un-\begin{matrix}  \\[-0.5cm]  \racine \end{matrix}\ $. The set $\wt G$ of characters of the bialgebra $\wt{\Cal H}$ forms a monoid for the convolution product. The set $G$ of characters of the Hopf algebra $\Cal H$ is a group, and can be seen as the submonoid of $\wt G$ formed by the elements $\varphi$ such that $\varphi(\! \begin{matrix}  \\[-0.5cm]  \racine \end{matrix})=1$.


\subsection{The associated pre-Lie structure}
\label{ssect:pre-lie}

Denote by $(Z_s)$ the dual basis in the graded dual $\Cal H^\circ$ of the forest basis of $\Cal H$. The correspondence $Z:s\mapsto Z_s$ extends linearly to a unique vector space isomorphism from $\Cal H$ onto $\Cal H^\circ$. For any tree $t$ the corresponding $Z_t$ is an infinitesimal character of $\Cal H$, i.e. it is a primitive element of $\Cal H^\circ$. We denote by $\star$ the (convolution) product of $\Cal H^\circ$. From the formula directly stemming from the definition:
\begin{equation*}
	(Z_t\star Z_u-Z_u\star Z_t)(v) = \sum_{s}Z_t(s)Z_u(v/s)-\sum_{s'}Z_u(s')Z_t(v/s')
\end{equation*}
we see that if $t$ and $u$ are trees, the two sums on the right-hand side run over classes of subtrees only (because $Z_t$ and $Z_u$ vanish on forests containing more than one tree). Looking more closely at it we arrive at the following:
\begin{equation*}
	Z_t\star Z_u-Z_u\star Z_t = Z_{t\rhd u-u\rhd t}.
\end{equation*}
Here $t\rhd u$ is obtained by inserting $t$ inside $u$, namely:
\begin{equation*}
	t\rhd u=\sum_{v,\,t\subset v \smop{and} v/t=u}N(t,u,v)v,
\end{equation*}
where $N(t,u,v)$ is the number of subtrees of $v$ isomorphic to $t$ such that $v/t$ is isomorphic to $u$. By the Cartier--Milnor--Moore theorem, the graded dual $\Cal H^\circ$ is isomorphic as a Hopf algebra to the enveloping algebra $\Cal U(\g g)$, where $\g g=\mop{Prim}\Cal H^\circ$ is the Lie algebra spanned by the $Z_t$'s for rooted trees $t$. The product $\rhd$ satisfies the left pre-Lie relation (\ref{prelie}). This pre-Lie structure can of course be transported on $\g g$ by setting $Z_t\rhd Z_u:=Z_{t\rhd u}$, and the Lie bracket is given by:
\begin{equation*}
	[Z_t,\,Z_u] = Z_t\star Z_u-Z_u\star Z_t
	                   =Z_t\rhd Z_u-Z_u\rhd Z_t.
\end{equation*}


\subsection{Another normalization}
\label{ssect:normalisation}

Denote by $A_\sigma:\Cal H\to\Cal H$ the linear map defined by $A_\sigma(s)=\sigma(s)s$ for any rooted forest $s$, where $\sigma(s)$ is the internal symmetry factor (see section \ref{sect:rt}). The map $A_\sigma$ is an algebra automorphism thanks to the multiplicativity of $\sigma$. We modify the coproduct according to $A_\sigma$, for each tree $t$ we define:
\begin{equation*}
	\Delta_\sigma(t):=(A_\sigma\otimes A_\sigma)\circ\Delta\circ A_\sigma^{-1}(t)
	=\sum_{s}\frac{\sigma(s)\sigma(t/s)}{\sigma(t)}s\otimes t/s,
\end{equation*}
where the sum runs over the subforests of $t$. Then $\Cal H$ endowed with this coproduct (without changing the product) is a Hopf algebra with antipode $S_\sigma=A_\sigma\circ S\circ A_\sigma^{-1}$. The associated pre-Lie structure on the graded dual $\Cal H^\circ$ is given by $Z_t\rhd_\sigma Z_u=Z_{t\rhd_\sigma u}$, with:
\begin{equation*}
	t\rhd_\sigma u=\sum_{v,\,t\subset v\smop{and} v/t=u}M(t,u,v)v,
\end{equation*}
where $M(t,u,v)=\displaystyle\frac{\sigma(t)\sigma(u)}{\sigma(v)}N(t,u,v)$ can be interpreted as the number of ways to insert the tree $t$ inside the tree $u$ in order to get the tree $v$. See \cite{H2} for more on the combinatorics of rooted trees and Hopf algebras.


\section{Some coproduct computations}
\label{sect:cocomp}

We will denote by $E_n$ the $n$-edged ladder (hence containing $n+1$ vertices), and by $C_n$ the $n$-edged corolla, i.e. the tree built from $n$ vertices each linked by one edge to a common root. $E_0$ and $C_0$ both coincide
with the unit $\begin{matrix}  \\[-0.5cm]  \racine \end{matrix}$.

\begin{prop}
\label{deltacorolla}
	$\Delta(C_n)=\sum_{p=0}^n{n\choose p}C_p\otimes C_{n-p}$ and $\Delta_\sigma(C_n)=\sum_{p=0}^nC_p\otimes C_{n-p}$.
\end{prop}

\begin{proof}
The only subforests of $C_n$ are the corollas $C_p,\,0\le p\le n$, each
appearing ${n\choose p}$ times, and $C_n/C_p=C_{n-p}$. The second equality
comes then from the obvious formula $\sigma(C_k)=k!$ for any $k$.
\end{proof}

\begin{cor}
Let $n\geq1$. Then
\begin{eqnarray}
S(C_n)&=&\sum_{1\leq r\leq n}(-1)^r\sum_{k_1+\cdots+k_r=n}\frac{n!}{k_1!\cdots k_r!}C_{k_1}\cdots
C_{k_r},\label{Scorolla}\\
S_\sigma(C_n)&=&\sum_{1\leq r\leq n}(-1)^r\sum_{k_1+\cdots+k_r=n}C_{k_1}\cdots C_{k_r}.
\label{Scorolla2}
\end{eqnarray}
\end{cor}

Computing $\Delta(E_n)$ needs some combinatorics: for any integer $n\ge 1$ and
for $1\le r\le n$ consider the set $K_{n,r}$ of the length $r$ compositions of
the integer $n$, namely:
\begin{equation*}
	K_{n,r}:=\{(p_1,\ldots,p_r)\in(\N-\{0\})^r,\, p_1+\cdots +p_r=n\}.
\end{equation*}
Let us call {\sl block\/} of the composition $\pi$ one of the intervals~:
\begin{equation*}
	P_j(\pi):=\{p_1+\cdots +p_j+1,\ldots,p_1+\cdots +p_{j+1}\},\,j=0,\ldots,r-1.
\end{equation*}

\begin{prop}
\begin{equation}
\label{eq-comp}
	\Delta(E_n)=\sum_{r\ge 1}\sum_{\pi= (p_1,\ldots,p_r)\in K_{n,r}}
				  \left(\prod_{j\ge 1}E_{p_{2j}}\right)\otimes E_{\sum_{j\ge 1}p_{2j-1}}
				 +\left(\prod_{j\ge 1}E_{p_{2j-1}}\right)\otimes E_{\sum_{j\ge 1}p_{2j}}.
\end{equation}
\end{prop}

\begin{proof}
Let us label the edges of $E_n$ from $1$ to $n$, upwards starting from the
root. Any interval of $\{1,\dots,n\}$ corresponds bijectively to a subtree of
$E_n$. Any composition $\pi$ of $n$ gives rise to exactly two subforests of
$E_n$: the one the subtrees of which are given by the blocks of $\pi$ of odd
rank, and the one the subtrees of which are given by the blocks of $\pi$ of even
rank. Any subforest is obtained this way once and only once, and any class of
subforests contains only one element (this coming from the fact that the
automorphism group of a ladder is trivial), which proves the result.
\end{proof}

The best way to get used to this particular coproduct is to present some examples. Here for the first ladders up to degree 5:
\begin{eqnarray*}
	\Delta(\! \racine )&=&\racine\otimes \racine\\
	\Delta\big( \arbrea\big)&=&\arbrea\otimes\racine+\racine\otimes \arbrea\\
	\Delta\big( \arbreba\big)&=&\arbreba\otimes \racine+\racine\otimes \arbreba+2\arbrea\otimes\arbrea\\
	\Delta\big( \arbreca\big)&=&\arbreca\otimes \racine+\racine\otimes \arbreca+2\arbreba\otimes\arbrea+3\arbrea\otimes \arbreba
						+\arbrea\arbrea\otimes \arbrea\\
	\Delta\big(\arbreda\big)&=&\arbreda\otimes\racine+\racine\otimes\arbreda
						+2\arbreca\otimes\arbrea+3\arbreba\otimes\arbreba+4\arbrea\otimes\arbreca
						+3\arbrea\arbrea\otimes\arbreba+2\arbreba\arbrea\otimes\arbrea.
\end{eqnarray*}
Note that the coproduct $\Delta (E_n)$ is the sum of $2^n$ elements: this comes
from the fact that the cardinal of $K_{n,r}$ is equal to ${n-1\choose r-1}$,
yielding $2^{n-1}$ for the total number of compositions of $n$. This suggests
a rewriting of the coproduct formula for $E_n$ by means of compositions of
$n+1$, which can be done the following way: there is a one-to-one
correspondence between $K_{n+1,r}$ and the set $\wt K_{n,r}$ of ``mock-compositions''
of $n$, namely sequences $(p_1,\ldots ,p_r)$ of nonnegative integers with
$p_1+\cdots +p_r=n$, all positive except perharps $p_1$. Equation
\eqref{eq-comp} can be rewritten as:
\begin{equation}
\label{eq-comp2}
	\Delta(E_n)=\sum_{r\ge 1}\sum_{\pi= (p_1,\ldots,p_r)\in \wt K_{n,r}}
	\left(\prod_{j\ge 1}E_{p_{2j-1}}\right)\otimes E_{\sum_{j\ge 0}p_{2j}}.
\end{equation}
Note that the two coproducts $\Delta_\sigma$ and $\Delta$ coincide on the
ladders $E_n$.
\begin{rmk}{\rm{
The coproduct in the bialgebra $\wt{\Cal H}$ shows forests involving the
single root $\begin{matrix}  \\[-0.5cm]  \racine \end{matrix}$ on the left, for example:
\begin{equation*}
	\Delta_{\wt{\Cal H}}\big( \arbreba\big)=\arbreba\otimes \racine
											+\racine\racine\racine\otimes \arbreba
											+2\arbrea\racine\otimes\arbrea.
\end{equation*}}}
\end{rmk}


\section{Another expression for the coproduct}
\label{sect:Delta2}

We try in this section to generalize formula  \eqref{eq-comp2} to any rooted tree. For this purpose we need the notion of composition for a tree (more exactly an analogue of the ``mock-compositions'' of the previous section), which will be derived from the concept of {\sl floored tree\/}.

\bigskip

Let $t$ be a tree different from the unit $\begin{matrix}  \\[-0.5cm]  \racine
\end{matrix}\hspace{0.1cm}\! $. Denote by $\gmop{ht}$ the height function, defined on the set
$E(t)$ of edges with values into the positive integers, by the distance from
the top vertex of the given edge down to the root. We consider the natural
partial order on $E(t)$, defined by $e\le e'$ if and only if there is a path
from the root flowing through $e$ and ending at $e'$.

\begin{defn}
A {\rm floored tree\/} is a tree together with a nondecreasing function ``floor''
$\gmop{fl}:E(t)\to\mathbb{N}$ such that for any edge $e$ the inequality
$\gmop{fl}(e)\le\gmop{ht}(e)$ holds, and such that the image of $\gmop{fl}$ is an interval (thus starting
with $0$ or $1$).
\end{defn}

Note that this is a kind of ``french" definition, as the lowest floor (which however can be empty) is numbered by zero. Observe that there is a natural one-to-one correspondence between floored trees with underlying tree $E_n$ and the mock-compositions of the integer $n$: to any such floored tree we associate the sequence $(k_1,\ldots ,k_r)$ where $r-1$ is the maximum of the function $\gmop{fl}$, and where $k_1+\cdots +k_j$ is the height reached by the $(j-1)^{th}$ floor.\\

Let us denote by $\wt K_r(t)$ the set of floored trees with underlying tree $t$ and with top floor of rank $r-1$. For any $\wt t\in \wt K_r(t)$ we denote by $\big(s_1(\wt t),\ldots s_r(\wt t)\big)$ the associated collection of subforests, namely $s_j(\wt t):=\gmop{fl}^{-1}(j-1)$. We have the following formula for the coproduct (where $h(t)$ stands for the maximum of the height function $\gmop{ht}$ on $t$):

\begin{prop}
\begin{equation}
\label{eq-delta}
	\Delta(t)=\sum_{r=1}^{h(t)}\sum_{\wt t\in \wt K_r(t)}
			  \left(\prod_{j\ge 1}s_{2j-1}(\wt t\,)\right)\otimes\, {t}\Big\slash{\prod_{j\ge 1}s_{2j-1}(\wt{t}\, )}.
\end{equation}
\end{prop}
As an example, a floored tree structure on the corolla $C_n$ is just a function $\gmop{fl}:\{e_1,\ldots e_n\}\to\{0,1\}$ where the edges are denoted by $e_j,\,j=1,\ldots ,n$. Formula \eqref{eq-delta} gives then back Proposition \ref{deltacorolla}.

\section{A formula for the antipode}
\label{sect:antipo}
\begin{prop}
The antipode of a tree $t$ in $\Cal H$ is given by:
\begin{equation}
S(t)=-t+ \sum_{r\ge 1}(-1)^{r+1}\sum_{\emptyset \subsetneq s_1\subsetneq...\subsetneq s_r\subsetneq t}
 s_1(s_2/s_1)...(s_r/s_{r-1})(t/s_r),
\end{equation}
where the $s_j$'s are subforests of $t$. 
\end{prop}
\begin{proof}
The proof is straightforward by applying repeatedly the recursive formula \eqref{eq-anti-recu} for the antipode, owing to the fact that for any subforest $\overline {s'}$ of $t/s$ there is a unique subforest $s'$ of $t$ which contains $s$ and such that $s'/s=\overline {s'}$.
\end{proof}

\section{The backwards error analysis character}
\label{sect:log}

Denote by $E_\sigma$ the character of $\Cal H$ given by: 
\begin{equation*}
	E_\sigma=\frac{\mop{CM}(t)}{v(t)!}\,,
\end{equation*}
where $\mop{CM}(t)$ stands for the Connes--Moscovici coefficient  (\ref{CM}) of the tree $t$ (e.g. see \cite{Br,C, F,KreimerChen}). In particular:
\begin{equation*}
	E_\sigma(C_n)=E_\sigma(E_n)=\frac{1}{(n+1)!}\,.
\end{equation*}
Now define, as in \cite{C}, $L_\sigma:=E_\sigma^{ -1}=E_\sigma\circ S_\sigma$. The inverse is understood with respect to the convolution product $\star_\sigma$ associated with the coproduct $\Delta_\sigma$ of paragraph \ref{ssect:normalisation}\footnote{The characters $E_\sigma$ and $L_\sigma$ are denoted respectively by $\exp^*$ and $\log^*$ by F. Chapoton in \cite{C}.}. As rooted trees form a natural basis of the free pre-Lie algebra in one generator (\cite{ChaLiv}, \cite{DL}), the values of $L_\sigma$ show up naturally in the Magnus expansion, i.e. in the expansion of the logarithm of the solution of the differential equation $\dot X = AX$, say, e.g. in an algebra of matrix valued functions (\cite{Mag}, \cite{GKLLRT}). This becomes more transparent when realizing its underlying pre-Lie algebra structure, set up in \cite{EM,EMdendeq}. In particular, according to \eqref{Scorolla2} we find: 
\begin{equation*}
	L_\sigma(C_n)=\sum_{1\leq r\leq n}(-1)^r\sum_{k_1+\cdots+k_r=n}\Big(\prod_{1\leq i\leq r}\frac{1}{(k_i+1)!}\Big)\,.
\end{equation*}
Hence, with the convention $C_0=\begin{matrix}  \\[-0.5cm]  \racine \end{matrix}$ we have:
\begin{eqnarray*}
\sum_{n\geq0}L_\sigma(C_n)x^n 
& = & 
1+\sum_{n\geq1}x^n\sum_{1\leq r\leq n}(-1)^r\sum_{k_1+\cdots+k_r=n}\Big(\prod_{1\leq i\leq r}\frac{1}{(k_i+1)!}\Big) \\
& = & 
1+\sum_{n\geq1}\sum_{1\leq r\leq n}(-1)^r\sum_{k_1+\cdots+k_r=n}\Big(\prod_{1\leq i\leq r}\frac{x^{k_i}}{(k_i+1)!}\Big) \\
& = & 
\sum_{r\geq0}(-1)^r\Big(\sum_{k\geq1}\frac{x^k}{(k+1)!}\Big)^r=\frac{x}{e^x-1}\,. 
\end{eqnarray*}

We recover then \cite[Proposition 10]{C}, namely $L_\sigma(C_n)=\frac{B_n}{n!}$, where the $B_n$'s stand for the Bernoulli numbers:
\begin{equation*}
	B_0=1,\ B_1=-\frac 12,\ B_2=\frac 16,\ B_4=-\frac{1}{30},\ldots,\ B_{2k+1}=0 \hbox{ for }k\ge 1.
\end{equation*}

Translating this into the first normalization, we may consider the character $E:=E_\sigma\circ A_\sigma$, given by:
\begin{equation*}
	E(t)=\frac 1{t!}.
\end{equation*}
Then its inverse $E^{-1}=E\circ S$ with respect to the convolution product associated with the original coproduct $\Delta$ is given by $L=L_\sigma\circ A_\sigma$. In particular we get:
\begin{equation*}
	L(C_n)=B_n.
\end{equation*}

\begin{rmk} \label{rmk:B-series}{\rm{
Recall the notion of $B$-series  \cite{Butcher1, Butcher2, HWL} as a formal power series in the step size parameter $h$ containing elementary differentials and arbitrary coefficients encoded in a linear function $\alpha$ on the set of rooted trees $T$:
 \allowdisplaybreaks{
\begin{eqnarray}
\label{Bseries1}
	B(\alpha, h a) = \alpha(\un)\un + \sum_{t \in T} h^{v(t)} \frac{\alpha(t)}{\sigma(t)} F_a(t),
\end{eqnarray}}
where $a$ is a smooth vector field on $\R^n$ and the elementary differential is defined as:
$$
	F_a(B_+(t_1 \cdots t_n))(y) = a^{(n)}(y) \bigl(F_a(t_1)(y),\ldots,F_a(t_n)(y) \bigr).
$$ 
Since Cayley \cite{Cay} we use rooted trees to encode the solution of autonomous initial value problem $\dot y(s) = a(y(s))$, $y(0)=y_0$, corresponding to (\ref{Bseries1}) with $\alpha: T \to k$:
$$
	\alpha(t):=E(t).
$$
For a detailed exposition of the ideas of backward error analysis and modified equations appearing in the context of numerical methods for differential equations and related algebraic structures, i.e. composition and (vector field-)substitution of $B$-series, we refer the reader to \cite{CHV,CHVj,CHV09,HWL}. In section \ref{sect:B-series} below we will show how the substitution law for $B$-series is related to the bialgebra $\wt{\Cal H}$. }}
\end{rmk}


\section{Relation with the Connes--Kreimer Hopf algebra}
\label{sect:ckhopf}


\subsection{The Connes--Kreimer algebra of rooted trees}

Let $k$ be a field of characteristic zero. The Connes--Kreimer Hopf algebra ${\mathcal H}_{\makebox{{\tiny{CK}}}}=\bigoplus_{n\geq 0} {\mathcal H}_{\makebox{{\tiny{CK}}}}^{(n)}$ is the Hopf algebra of rooted forests over $k$, graded by the number of vertices. It is the free commutative algebra on the linear space $T$ spanned by nonempty rooted trees. For a rooted tree $t$, we denote by $E(t)$, $V(t)$ the set of edges and vertices, respectively. The coproduct on a rooted forest $u$ (i.e. a product of rooted trees) is described as follows: the set $U$ of vertices of a forest $u$ is endowed with a partial order defined by $x \le y$ if and only if there is a path from a root to $y$ passing through $x$. Any subset $W$ of the set of vertices $U$ of $u$ defines a {\sl subforest\/} $w$ of $u$ in an obvious manner, i.e. by keeping the edges of $u$ which link two elements of $W$. The coproduct is then defined by:
\begin{equation}
\label{coprod}
	\Delta_{\makebox{{\tiny{CK}}}}(u)= \sum_{V \amalg W=U \atop W<V}v\otimes w.
\end{equation}
Here the notation $V<W$ means that $x<y$ for any vertex $x$ of $v$ and any vertex $y$ of $w$ such that $x$ and $y$ are comparable. Such a couple $(V,W)$ is also called an {\sl admissible cut\/}, with crown (or pruning) $v$ and trunk $w$. For comparison we remind the reader of the Connes--Kreimer coproduct defined in terms of admissible cuts on a rooted tree $t$, i.e. $c \subset E(t)$:
 \begin{equation}
    \Delta_{\makebox{{\tiny{CK}}}}(t) = t \otimes \un + \un \otimes t + \sum_{c \in {\tiny{\mop{Adm}(t)}}} P^c(t)\otimes R^c(t).
    \label{coprod1}
 \end{equation}
Here we denote by $\mop{Adm}(t)$ the set of admissible cuts of a forest, i.e. the set of collections of edges such that any path from the root to a leaf contains at most one edge of the collection. In order to make this picture completely correct, we must stress that for any nonempty tree two particular admissible cuts must be associated with the empty collection: the empty cut and the total cut, or digging out. Following \cite{F} we denote as usual by $P^c(t)$ (resp. $R^c(t)$) the pruning (resp. the trunk) of $t$, i.e. the subforest formed by the edges above the cut $c \in \mop{Adm}(t)$ (resp. the subforest formed by the edges under the cut). Note that the trunk of a tree is a tree, but the pruning of a tree may be a forest. Here, $\un$ stands for the empty forest, which is the unit. One sees easily, that $\deg(t) = \deg(P^c(t)) + \deg(R^c(t))$, for all admissible cuts. We present two examples:
 \allowdisplaybreaks{
\begin{eqnarray*}
 \Delta_{\makebox{{\tiny{CK}}}}\big(\arbrea\big) &=&
                        \arbrea \otimes \un
                            + \un \otimes \arbrea
                               + \racine \otimes \racine \\
 \Delta_{\makebox{{\tiny{CK}}}}\big(\! \arbrebb \big) &=& 
 					\arbrebb \otimes \un
                               + \un \otimes \arbrebb +
                                  2\racine \otimes\arbrea
                                   + \racine\racine\otimes \racine \, .
\end{eqnarray*}}
Observe that  the coproduct $\Delta_{\makebox{{\tiny{CK}}}}$ is highly non-cocommutative.

With the restriction that $V$ and $W$ be nonempty (i.e. if $V$ and $W$ give rise to an ordered partition of $U$ into two blocks) we get the restricted coproduct:
\begin{equation}
\label{coprod2}
	\wt\Delta_{\makebox{{\tiny{CK}}}}(u)=\Delta_{\makebox{{\tiny{CK}}}}(u)-u\otimes\un -\un\otimes u
									       =\sum_{V\amalg W=U \atop W<V,\, V,W\not =\emptyset}v\otimes w,
\end{equation}
which is often displayed $\sum_{(u)} u' \otimes u''$ in Sweedler's notation. The iterated restricted coproduct writes in terms of ordered partitions of $U$ into $n$ blocks:
\begin{equation}
\label{iter-coprod}
	\wt\Delta_{\makebox{{\tiny{CK}}}}^{n-1}(u)=\sum_{V_1\amalg\cdots\amalg V_n=U \atop V_n<\cdots <V_1,\, V_j\not 
  											   =\emptyset}v_1\otimes\cdots\otimes v_n,
\end{equation}
and we get the full iterated coproduct $\Delta_{\makebox{{\tiny{CK}}}}^{n-1}(u)$ by allowing empty blocks in the formula above.\\

Let $B_+$ be the operator (of degree one) acting on a forest $u=t_1\cdots t_n$ by grafting the components to a new common root. Recall that $B_+$ is a coalgebra Hochschild cocycle, namely for any forest $u$:
\begin{equation*}
	\Delta_{\makebox{{\tiny{CK}}}}\circ B_+(u)=B_+(u)\otimes \un+(\mop{Id} \otimes B_+)\circ\Delta_{\makebox{{\tiny{CK}}}}(u),
\end{equation*}
where $\un:k \to {\mathcal H}_{\makebox{{\tiny{CK}}}}$ is the unit (see \cite{F}). Note (see e.g. \cite{H2,P}) that ${\mathcal H}_{\makebox{{\tiny{CK}}}}$ is isomorphic, as a Hopf algebra, to the graded dual of the Grossman--Larson Hopf algebra.\\

There is also an associated pre-Lie structure: Denote by $(\delta_s)$ the dual basis in the graded dual $\Cal H_{\makebox{{\tiny{CK}}}}^\circ$ of the forest basis of $\Cal H_{\makebox{{\tiny{CK}}}}$. The correspondence $\delta:s\mapsto \delta_s$ extends linearly to a unique vector space isomorphism from $\Cal H_{\makebox{{\tiny{CK}}}}$ onto $\Cal H_{\makebox{{\tiny{CK}}}}^\circ$. For any tree $t$ the corresponding $\delta_t$ is an infinitesimal character of $\Cal H_{\makebox{{\tiny{CK}}}}$, i.e. it is a primitive element of $\Cal H^\circ$. We denote by $\ast$ the (convolution) product of $\Cal H^\circ$. We have:
\begin{equation*}
	\delta_t\ast \delta_u-\delta_u\ast \delta_t=\delta_{t\to u-u\to t}.
\end{equation*}
Here $t\to u$ is obtained by grafting $t$ on $u$, namely:
\begin{equation*}
	t\to u=\sum_{v}N'(t,u,v)v,
\end{equation*}
where $N'(t,u,v)$ is the number of partitions $V(t)=V\amalg W$, $W<V$ such that $v\restr{V}=t$ and $v\restr{W}=u$. By the Cartier--Milnor--Moore theorem, the graded dual $\Cal H_{\makebox{{\tiny{CK}}}}^\circ$ is isomorphic as a Hopf algebra to the enveloping algebra $\Cal U(\g g_{\makebox{{\tiny{CK}}}})$, where $\g g_{\makebox{{\tiny{CK}}}}=\mop{Prim}\Cal H_{\makebox{{\tiny{CK}}}}^\circ$ is the Lie algebra spanned by the $\delta_t$'s for rooted trees $t$. The product $\to$ satisfies the left pre-Lie relation (\ref{prelie}). This pre-Lie structure can of course be transported on $\g g_{\makebox{{\tiny{CK}}}}$ by setting $\delta_t\to \delta_u:=\delta_{t\to u}$, and the Lie bracket is given by:
\begin{equation*}
	[\delta_t,\,\delta_u]=\delta_t\ast \delta_u-\delta_u\ast \delta_t=\delta_t\to \delta_u-\delta_u\to \delta_t.
\end{equation*}
Another normalization is often employed, by setting:
\begin{equation*}
	t\curvearrowright u=A_\sigma^{-1}\big(A_\sigma(t)\to A_\sigma(u)\big).
\end{equation*}
We have immediately:
\begin{equation*}
	t\curvearrowright u=\sum_{v}M'(t,u,v)v,
\end{equation*}
where $M'(t,u,v)=\displaystyle\frac{\sigma(t)\sigma(u)}{\sigma(v)}N'(t,u,v)$ can be interpreted as the number of ways to graft the tree $t$ on the tree $u$ in order to get the tree $v$. Considering the normalized dual basis $\wt\delta_t=\sigma(t)\delta_t$ we obviously have:
\begin{equation*}
	\wt\delta_t\ast \wt\delta_u-\wt\delta_u\ast \wt\delta_t=\wt\delta_{t\curvearrowright u-u\curvearrowright t}.
\end{equation*}

The pre-Lie algebra $(T,\curvearrowright)$ is the free pre-Lie algebra with one generator \cite{ChaLiv}. For more properties of the two pre-Lie structures $\to$ and $\rhd$ (or equivalently $\curvearrowright$ and $\rhd_\sigma$), see \cite{MS}.


\subsection{A compatible left comodule structure}

One observes that there is a $\Cal H$-bicomodule structure on the Connes--Kreimer Hopf algebra $\Cal H_{\makebox{{\tiny{CK}}}}$ defined as follows: for any nonempty tree $t$ we set $\Phi(t)=\Psi(t)=\Delta(t)$, for the unit tree $\un$ we set $\Phi(\un)=\racine\otimes\un$,  $\Psi(\un)=\un\otimes\racine$, and we extend $\Phi$ (resp. $\Psi$) to an algebra morphism from $\Cal H_{\makebox{{\tiny{CK}}}}$ to $\Cal H\otimes\Cal H_{\makebox{{\tiny{CK}}}}$ (resp.  $\Cal H_{\makebox{{\tiny{CK}}}}\otimes\Cal H$). The coaction axioms for $\Phi$ and $\Psi$ are clearly verified as well as the compatibility condition $(\Phi\otimes \mop{Id}_{\Cal H})\circ\Psi=(\mop{Id}_{\Cal H}\otimes\Psi)\circ\Phi$. We will use Sweedler's notation:
\begin{equation*}
	\Phi(x)=\sum_{(x)}x_1\otimes x_0, \hskip 20mm\Psi(x)=\sum_{(x)}x'_0\otimes x'_1.
\end{equation*}
We are now interested in finding relations between this bicomodule structure and the Connes--Kreimer coproduct $\Delta_{\makebox{{\tiny{\rm{CK}}}}}$.

\begin{thm}\label{thm:comp}
The following identity of linear maps from $\Cal H_ {\makebox{{\tiny{\rm{CK}}}}}$ into $\Cal H\otimes\Cal H_{\makebox{{\tiny{\rm{CK}}}}}\otimes\Cal H_{\makebox{{\tiny{\rm{CK}}}}}$ holds:
\begin{equation}
\label{codistrib}
	(\mop{Id}_{\Cal H}\otimes\Delta_{\makebox{{\tiny{\rm{CK}}}}})\circ\Phi=m^{1,3}\circ(\Phi\otimes\Phi)\circ\Delta_{\makebox{{\tiny{\rm{CK}}}}},
\end{equation}
where $m^{1,3}:\Cal H\otimes\Cal H_{\makebox{{\tiny{\rm{CK}}}}}\otimes\Cal H\otimes\Cal H_{\makebox{{\tiny{\rm{CK}}}}}
\longrightarrow \Cal H\otimes\Cal H_{\makebox{{\tiny{\rm{CK}}}}}\otimes\Cal H_{\makebox{{\tiny{\rm{CK}}}}}$ is defined by:
\begin{equation*}
	m^{1,3}(a\otimes b\otimes c\otimes d)=ac\otimes b\otimes d.
\end{equation*}
\end{thm}

\begin{proof}The verification is immediate for the empty forest. Recall that we denote by $\mop{Adm}(t)$ the set of admissible cuts of a forest. We have then for any nonempty forest:
 \allowdisplaybreaks{
\begin{eqnarray*}
(\mop{Id}_{\Cal H}\otimes\Delta_{\makebox{{\tiny{CK}}}})\circ\Phi\big(t)&=&(\mop{Id}_{\Cal H}\otimes\Delta_{\makebox{{\tiny{CK}}}})
						\sum_{s\smop{ subforest}\atop \smop{ of} t}s\otimes t/s\\
	&=&\sum_{s\smop{subforest}\atop \smop{of}t}\ \sum_{c\in\smop{Adm}(t/s)}s\otimes P^c(t/s)\otimes R^c(t/s).
\end{eqnarray*}}
On the other hand we compute:
\allowdisplaybreaks{
\begin{eqnarray*}
m^{1,3}\circ(\Phi\otimes\Phi)\circ\Delta_{\makebox{{\tiny{CK}}}}(t)&=&
	m^{1,3}\circ(\Phi\otimes\Phi)\left(\sum_{c\in\smop{Adm}(t)}P^c(t)\otimes R^c(t)\right)\\
	&\hskip -60mm = &\hskip -30mm 
	m^{1,3}\left(\sum_{c\in\smop{Adm}(t)}\ \sum_{s'\smop{subforest}\atop \smop{ of}P^c(t)}\ 
	\sum_{s''\smop{subforest}\atop \smop{of}R^c(t)}
	\hskip -2mm  s'\otimes P^c(t)/s'\otimes s''\otimes R^c(t)/s''\right)\\
	&\hskip -60mm =&\hskip -30mm 
	\sum_{c\in\smop{Adm}(t)}\ \sum_{s'\smop{subforest}\atop \smop{ of}P^c(t)}\ 
	\sum_{s''\smop{subforest}\atop \smop{ of}R^c(t)}
	 \hskip -2mm s's''\otimes P^c(t)/s'\otimes R^c(t)/s''\\
	&\hskip -60mm =&\hskip -30mm 
	\sum_{c\in\smop{Adm}(t)}\ \sum_{s\smop{subforest of}t \atop \smop{ containing no edge of}c}
	\hskip -7mm s\otimes P^c(t)\Big\slash s\cap P^c(t)\otimes R^c(t)\Big\slash s \cap R^c(t)\\
	&\hskip -60mm =&\hskip -30mm 
	\sum_{s\smop{subforest}\atop \smop{ of}t}\ \sum_{c\in\smop{Adm}(t/s)}s\otimes P^c(t/s)\otimes R^c(t/s),
\end{eqnarray*}}
which proves the theorem.
\end{proof}

\begin{cor}\label{transposeL}
Let $a:\Cal H\to k$ be any linear map. Then the operator
${}^t\!L_a :\Cal H_{\makebox{\tiny{\rm{CK}}}}\to \Cal H_{\makebox{\tiny{\rm{CK}}}}$ defined by
${}^t\!L_a=(a\otimes \mop{Id}_{ \Cal H_{\makebox{\tiny{\rm{CK}}}}}) \circ\Phi$, i.e.:
\begin{equation*}
	^t\!L_a(x)=\sum_{(x)} \langle a,\,x_1 \rangle x_0
\end{equation*}
satisfies the identity:
\begin{equation*}
	\Delta_{\makebox{\tiny{\rm{CK}}}}\circ{}^t\!L_a={}^t\!L_{m^* a}\circ\Delta_{\makebox{\tiny{\rm{CK}}}},
\end{equation*}
where: $m^*:\Cal H^*\to (\Cal H\otimes\Cal H)^*$ is defined by
$m^*(a)(x\otimes y):=a(xy)$, and where
\begin{equation*}
{}^t\!L_{m^* a}:=(m^*a\otimes
  \mop{Id}_{\Cal H_{\makebox{\tiny{\rm{CK}}}}}\otimes
  \mop{Id}_{\Cal H_{\makebox{\tiny{\rm{CK}}}}})\circ\tau_{2,3}\circ(\Phi\otimes\Phi).
\end{equation*}
In particular when $a\in\Cal H^\circ$ then $m^*a=\sum_{(a)}a_1\otimes a_2\in\Cal H^\circ\otimes\Cal H^\circ$, and:
\begin{equation*}
	\Delta_{\makebox{\tiny{\rm{CK}}}} \circ{}^t\!L_a=\sum_{(a)}({}^t\!L_{a_1}\otimes{}^t\!L_{a_2})\circ \Delta_{\makebox{\tiny{\rm{CK}}}}.
\end{equation*}
\end{cor}

\begin{proof}
The operator ${}^t\!L_a$ is the transpose of the left multiplication operator $L_a:\Cal H_{\makebox{\tiny{\rm{CK}}}}^\circ\to\Cal H_{\makebox{\tiny{\rm{CK}}}}^\circ$ with respect to the left $\Cal H^\circ$-module structure (i.e. $L_a(b)=a\star b$), which justifies the notation. Similarly, if $a\in\Cal H^\circ$ then $m^*a\in \Cal H^\circ\otimes \Cal H^\circ$, and $^t\!L_{m^*a}$ is the transpose of the left multiplication operator $ L_{m^*a}:
\Cal H_{\makebox{\tiny{\rm{CK}}}}^\circ \otimes
\Cal H_{\makebox{\tiny{\rm{CK}}}}^\circ \to
\Cal H_{\makebox{\tiny{\rm{CK}}}}^\circ \otimes
\Cal H_{\makebox{\tiny{\rm{CK}}}}^\circ$ with respect to the left $\Cal
H^\circ\otimes\Cal H^\circ$-module structure given by $\wt\Phi=\tau_{2,3}\circ(\Phi\otimes\Phi)$. Here the notation $\tau_{2,3}$ stands for the flip of the two middle terms, namely $\tau_{2,3}(a\otimes b\otimes c\otimes d)=a\otimes c\otimes b\otimes d$. The proof of the corollary is a straightforward computation in view of Theorem \ref{thm:comp}:
 \allowdisplaybreaks{
\begin{eqnarray*}
\Delta_{ \makebox{\tiny{\rm{CK}}}} \circ{}^t\!L_a&=&
\Delta_{\makebox{\tiny{\rm{CK}}}}\circ (a\otimes\mop{Id}_{ \Cal H_{\makebox{\tiny{\rm{CK}}}}})\circ\Phi\\
					   &=&
(a\otimes\mop{Id}_{ \Cal H_{\makebox{\tiny{\rm{CK}}}}}\otimes\mop{Id}_{ \Cal H_{\makebox{\tiny{\rm{CK}}}}})\circ (\mop{Id}_{\Cal H}\otimes\Delta_{\makebox{\tiny{\rm{CK}}}})\circ\Phi\\
					   &=&(a\otimes\mop{Id}_{ \Cal H_{\makebox{\tiny{\rm{CK}}}}}\otimes\mop{Id}_{ \Cal H_{\makebox{\tiny{\rm{CK}}}}})\circ m^{1,3}\circ(\Phi\otimes\Phi)\circ\Delta_{\makebox{\tiny{\rm{CK}}}}\\
					   &=&(m^* a\otimes \mop{Id}_{ \Cal H_{\makebox{\tiny{\rm{CK}}}}}\otimes \mop{Id}_{ \Cal H_{\makebox{\tiny{\rm{CK}}}}})\circ \tau_{2,3}\circ(\Phi\otimes\Phi)\circ\Delta_{\makebox{\tiny{\rm{CK}}}}\\
					   &=&{}^t\!L_{m^* a}\circ\Delta_{\makebox{{\tiny{CK}}}}.
\end{eqnarray*}}

\end{proof}
Note that a similar property for the right coaction operator ${}^t\!R_a$ is not available, due to the fact that  the coproduct $\Delta_{\makebox{{\tiny{CK}}}}$ is highly non-cocommutative.

\begin{prop}\label{biderivation}
Let $a:\Cal H\to k$ be an infinitesimal character of $\Cal H$. Then the operator $^t\!L_a$ is a biderivation of the Hopf algebra $\Cal H_{\makebox{\tiny{\rm{CK}}}}$. Similarly if $\varphi:\Cal H\to k$ is a character of $\Cal H$, the operator $^t\!L_\varphi$ is an automorphism of the Hopf algebra $\Cal H_{\makebox{\tiny{\rm{CK}}}}$.
\end{prop}

\begin{proof}
Let $u=t_1\cdots t_n$ be any forest in $\Cal H_{\makebox{{\tiny{CK}}}}$ and $a$ an infinitesimal character of $\Cal H$. We have then:
 \allowdisplaybreaks{
\begin{eqnarray*}
{}^t\!L_a(u)&=&\sum_{(u)}\langle a,\,u_1 \rangle u_0\\
	&=&\sum_{s\smop{subforest of}u}\langle a,\,s\rangle\,u/s\\
	&=&\sum_{s\smop{subtree of}u}\langle a,\,s \rangle\,u/s\\
	&=&\sum_{j=1}^n\sum_{s\smop{subtree of}t_j}\langle a,\,s \rangle \,u/s\\
	&=&\sum_{j=1}^nt_1 \cdots t_{j-1}\,\!^t\!L_a(t_j)t_{j+1}\cdots t_n,
\end{eqnarray*}}
hence $^t\!L_a$ is a derivation. The coderivation property is an immediate
consequence of Corollary \ref{transposeL}. The statement for characters
follows from a similar computation.
\end{proof}

\ignore{
Remember (see e.g. \cite{P}) that $H_R$ is isomorphic, as a Hopf algebra to the dual $H_{GL}:=A_{GL}^\circ$ of the 
Grossman-Larson Hopf algebra $A_{GL}$. As a vector space $H_{GL}$ is generated by non empty rooted tress. The product 
between two rooted trees in $H_{GL}$ is simply obtain by identifying their roots (in, particular the unit is $\racine$). 
The isomorphism $\varphi:H_R\tilde\longrightarrow H_{GL}$ is defined as follows: for any forest $t_1\cdots t_n$, 
$\varphi(t_1\cdots t_n)$ is the rooted tree obtained by grafting each $t_i$ on a leaf of the corolla $C_n$. Its inverse 
$\varphi^{-1}$ sends a rooted tree to the forest obtained by removing the root. We write $\Delta_{GL}$ for the Kreimer 
coproduct $\Delta_R$ transported to $H_{GL}$ through $\varphi$. 

\medskip

Now observe that $H_{GL}$ is a vector subspace of $\mathcal H$ and that the coproduct $\Delta(t)$ of a rooted tree 
lies in $\mathcal H\otimes H_{GL}$. Then do we have the following analogue of \cite[Proposition 3.5]{CHV} ?
\begin{con}
For any rooted tree $t$, 
$$
	(\Delta_{GL}\otimes{\rm id}_{\mathcal H})\circ\Delta(t)=m^{2,4}\circ(\Delta\otimes\Delta)\circ\Delta_{GL}(t)
$$
(where $m^{2,4}(a\otimes b\otimes c\otimes d)=a\otimes c\otimes bd$) and 
$$
	\Delta\circ S_{GL}(t)=(S_{GL}\otimes{\rm id})\circ\Delta(t)
$$
\end{con}
}

\begin{cor}\label{automorphismes}
The compatibility between the Hopf algebra structure and the comodule structure yields:
\begin{enumerate}

\item For any infinitesimal character $a$ of $\Cal H$ the operator $L_a$ is a biderivation of $\Cal H_{\makebox{{\tiny{\rm{CK}}}}}^\circ$.

\item Any character $\varphi$ of $\Cal H$ defines an automorphism $L_\varphi$ of the Hopf algebra $\Cal H_{\makebox{{\tiny{\rm{CK}}}}}^\circ$.

\end{enumerate}
\end{cor}

\begin{proof}
This is a direct consequence of Proposition \ref{biderivation}.
\end{proof}

We denote by the same sign $\star$ the convolution product on $\Cal H^\circ$ and the left and right actions of $\Cal H^\circ$ on $\Cal H_{\makebox{{\tiny{\rm{CK}}}}}^\circ$. We keep the other star $*$ for the convolution product on $\Cal H_{\makebox{{\tiny{\rm{CK}}}}}^\circ$. For any forest $s\in\Cal H_{\makebox{{\tiny{\rm{CK}}}}}$ (resp. $\Cal H$) we denote by $\delta_s\in\Cal H_{\makebox{{\tiny{\rm{CK}}}}}^\circ$ (resp. $Z_s\in\Cal H^\circ$) the corresponding element of the dual basis. Corollary \ref{automorphismes} in particular implies:

\begin{cor}\label{chv3-5}
Let $\varphi$ be a character of $\Cal H$, le $\alpha$ be any linear map from
$\Cal H$ into $k$, and let $b,c$ be linear maps form $\Cal
H_{\makebox{{\tiny{\rm{CK}}}}}$ into $k$. Let $\varepsilon=\delta_{\emptyset}$ be the co-unit of $\Cal H_{\makebox{{\tiny{\rm{CK}}}}}$. Then:
 \allowdisplaybreaks{
\begin{eqnarray}
		 \alpha\star\varepsilon&=&\varepsilon\star\alpha=\alpha(\racine)\varepsilon,\label{CHV1}\\
	\alpha\star \delta_{\racine}&=&\delta_{\racine}\star \alpha=\alpha,\label{CHV2}\\
		         Z_{\racine}\star b&=&b\star Z_{\racine}=b,\label{CHV3}\\
			    \varphi\star(b*c)&=&(\varphi\star b)*(\varphi\star c),\label{CHV4}\\
		    (\varphi\star b)^{*-1}&=&\varphi\star b^{*-1}.\label{CHV5}
\end{eqnarray}}
\end{cor}

The formulation of this result is completely parallel to Proposition 3.5 in \cite{CHV} on composition and substitution of $B$-series: we will fully
justify this similarity and prove in the next section that Corollary \ref{chv3-5} and Proposition \ref{CHVsuite} (more precisely their counterparts
w.r.t. bialgebra $\wt{\Cal H}$ described in Section \ref{sect:bigebre} below) indeed imply  Proposition 3.5 in \cite{CHV}. Following A. Murua \cite{M} we denote by $\delta$ the character of the Connes--Kreimer Hopf algebra  $\Cal H_{\makebox{{\tiny{CK}}}}$ such that $\delta(\begin{matrix}  \\[-0.5cm]  \racine \end{matrix})=1$ and $\delta(t)=0$ for any tree different from $\emptyset$ and $\begin{matrix}  \\[-0.5cm]  \racine \end{matrix}$. This should not be confused with the infinitesimal character $\delta_{\racine}$.

\begin{prop}\label{correspondence}
There is a one-to-one correspondence between characters of $\Cal H$ and a distinguished subset of characters of $\Cal H_{\makebox{{\tiny{\rm{CK}}}}}$, namely:
\begin{enumerate}

\item For any character $\varphi$ of $\Cal H$, the element
  $\varphi':=\varphi\star\delta_{\racine}$ is the infinitesimal character of $\Cal H_{\makebox{{\tiny{\rm{CK}}}}}$
  which coincides with $\varphi$ on nonempty trees. This settles a one-to-one
  correspondence $\varphi \mapsto \varphi \star \delta_{\racine}$  between
  characters of $\Cal H$ and infinitesimal characters of $\Cal H_{\makebox{{\tiny{\rm{CK}}}}}$ which
  take the value $1$ on the tree $\begin{matrix}  \\[-0.5cm]  \racine \end{matrix}$.
  
\item For any character $\varphi$ of $\Cal H$, the element
  $\varphi_\delta:=\varphi\star\delta$ is the character of $\Cal H_{\makebox{{\tiny{\rm{CK}}}}}$
  which coincides with $\varphi$ on nonempty trees. This 
  settles a one-to-one correspondence $\varphi\mapsto 
  \varphi\star\delta$  between characters of $\Cal H$ and 
  characters of $\Cal H_{\makebox{{\tiny{\rm{CK}}}}}$ which take the value $1$ on 
  the tree $\begin{matrix}  \\[-0.5cm]  \racine \end{matrix}$.

\end{enumerate}
\end{prop}

\begin{proof}
Let $\varphi$ be a character of $\Cal H$. By using the fact that $\Phi$ is an
algebra morphism it is straightforward to check that $\varphi\star\delta$ is a
character of $\Cal H_{\makebox{{\tiny{\rm{CK}}}}}$, and that $\varphi\star\delta(t)=\varphi(t)$ for
any nonempty tree $t$. Replacing $\delta$ by $\delta_{\racine}$ we obviously get
an infinitesimal character. The rest is an immediate check.
\end{proof}

Using  Proposition \ref{correspondence} and identity \eqref{CHV5} we
immediately obtain the analogon of the last identity of \cite[Prop. 3.5]{CHV}:

\begin{prop}\label{CHVsuite}
Let $b$ be a character of $\Cal H_{\makebox{{\tiny{\rm{CK}}}}}$ such that $b(\racine)=1$, and let $\wt
b$ be the unique character of $\Cal H$ such that $b=\wt b\star\delta$. Then:
\begin{equation}
b^{*-1}=\wt b\star\delta^{*-1}.
\end{equation}
\end{prop}

We will denote by $\exp^*$ and $\log^*$ the exponential and the logarithm\footnote{Note that this is not related to the notations used by Chapoton in \cite{C}, briefly recalled in section \ref{sect:log}, which may be quite confusing here.} with respect to the convolution product of $\Cal H^\circ_{\makebox{{\tiny{CK}}}}$. Now using Proposition \ref{correspondence} and identity \eqref{CHV4} we see that any character $b$ of $\Cal H_{\makebox{{\tiny{CK}}}}$ with $b(\begin{matrix}  \\[-0.5cm]  \racine \end{matrix})=1$ writes in a unique way:
\begin{equation}
\label{corr2}
	b=\exp^*(\varphi\star\delta_{\racine} )=\varphi\star(\exp^* \delta_{\racine} ),
\end{equation}
where $\varphi$ is a character of $\Cal H$. Recall that:
\begin{equation*}
	\exp^* \delta_{\racine}(t)=\frac{1}{t!}
\end{equation*}
for any nonempty tree (see \cite[Theorem 9]{M} and the proof therein). In
other words, $\exp^* \delta_{\racine}=E\star\delta$, where $E$ is the
character defined in section \ref{sect:log}.

\subsection{The left $\wt{\Cal H}$-comodule structure}
\label{sect:bigebre}

All results of the previous subsection are still valid almost without modification when one replaces the
Hopf algebra $\Cal H$ with the bialgebra $\wt {\Cal H}$. Remark first that
$\wt {\Cal H}$ and $\Cal H_{\makebox{\tiny{\rm{CK}}}}$ are naturally isomorphic as vector spaces and
even as algebras. The comodule map $\Phi: \Cal H_{\makebox{\tiny{\rm{CK}}}}\to \wt {\Cal H}\otimes\Cal
H_{\makebox{\tiny{\rm{CK}}}}$ is simply given by the coproduct $\Delta$ modulo this
identification. Theorem \ref{thm:comp} is still valid in this setting, as well
as Corollary \ref{transposeL}, Proposition \ref{biderivation} and Corollary
\ref{automorphismes}. Corollary \ref{chv3-5} also holds provided $\racine$ is
replaced by the unit $\un$ in Equation
\eqref{CHV1}. Proposition \ref{correspondence} is replaced by:

\begin{prop}\label{correspondence2}
There is a one-to-one correspondence between characters of $\wt{\Cal H}$ and characters of $\Cal H_{\makebox{{\tiny{\rm{CK}}}}}$, namely:
\begin{enumerate}

\item For any character $\varphi$ of $\wt {\Cal H}$, the element
  $\varphi':=\varphi\star\delta_{\racine}$ is the infinitesimal character of $\Cal H_{\makebox{{\tiny{\rm{CK}}}}}$
  which coincides with $\varphi$ on nonempty trees. This settles a one-to-one
  correspondence $\varphi \mapsto \varphi \star \delta_{\racine}$  between
  characters of $\wt{\Cal H}$ and infinitesimal characters of $\Cal H_{\makebox{{\tiny{\rm{CK}}}}}$.

  \item For any character $\varphi$ of $\wt{\Cal H}$, the element
  $\varphi_\delta:=\varphi\star\delta$ is the character of $\Cal H_{\makebox{{\tiny{\rm{CK}}}}}$
  which coincides with $\varphi$ on nonempty trees. This 
  settles a one-to-one correspondence $\varphi\mapsto 
  \varphi\star\delta$  between characters of $\wt{\Cal H}$ and 
  characters of $\Cal H_{\makebox{{\tiny{\rm{CK}}}}}$. 
\end{enumerate}
\end{prop}

Concerning item (2) in the foregoing proposition it is easily seen that this correspondence identifies two characters via the natural algebra isomorphism between $\Cal H_{\makebox{\tiny{\rm{CK}}}}$ and $\wt{\Cal H}$.

Finally, Proposition \ref{CHVsuite} still holds with condition $b(\racine)=1$
dropped, where $\wt b$ is now the unique character of $\wt{\Cal H}$ such that $b=\wt b\star\delta$.

\subsection{Operadic interpretation}
\label{operads}

Recall \cite{C} that one can associate a pro-nilpotent group $G_O$ to any augmented operad $O$ (i.e. any operad with no $0$-ary operation and the identity as unique $1$-ary operation). The associated Lie algebra $\g g_O$ is in fact a pre-Lie algebra. The group of characters of $\Cal H_\sigma$ is exactly the group associated with the augmented operad PreLie (\cite{ChaLiv}, \cite{C}). The pre-Lie operation giving rise to the corresponding Lie algebra is exactly the pre-Lie operation $\rhd_\sigma$. One should emphasize that the existence of a pre-Lie structure on $\g g_{O}$ has nothing to do with the fact that the operad $O$ is PreLie itself! Our second pre-Lie structure is the one on the free pre-Lie algebra with one generator, which is nothing but the linear span of rooted trees endowed with the grafting  
$\curvearrowright$ (see next section).\\

One may wonder if the character group $G_{\makebox{{\tiny{CK}}}}$ of the
Connes--Kreimer Hopf algebra is obtained from an augmented operad $\mop{CK}$
along the same lines. A partial answer to this question is given by the NAP
operad (\cite{CL2}, \cite{L}). This operad is defined as follows:
$\mop{NAP}(n)$ is the linear span of rooted trees with $n$ vertices numbered
from $1$ to $n$. The action of the symmetric group $S_n$ is obvious, and the
composition
\begin{eqnarray*}
	\gamma:\mop{NAP}(n)\times \mop{NAP}(p_1)\times\cdots\times
	\mop{NAP}(p_n)&\longrightarrow& \mop{NAP}(p_1+\cdots + p_n)\\
	(t;t_1,t_2,\ldots,t_n)&\longmapsto&\gamma(t;t_1,t_2,\ldots, t_n)
\end{eqnarray*}
is given by replacing vertex number $i$ of $t$ by the root of $t_i$ and
re-indexing the vertices of the big tree thus obtained. NAP stands for
``Non-Associative Permutative''. A NAP algebra is a vector space $A$ endowed
with a bilinear map $\circ:A\times A\to A$ such that $(a\circ b)\circ
c=(a\circ c)\circ b$. NAP algebras already appeared under the name ``right
commutative algebras'' in \cite{DL}, where the authors also show that the free
NAP algebra with one generator is the space of rooted trees endowed with the
Butcher product (see below). It is proven in
\cite{CL2} that the character group of $\Cal H_{\makebox{{\tiny{CK}}}}$ is a
  subgroup of $G_{\makebox{{\tiny{NAP}}}}$ in a natural way.

\section{$B$-series: composition and substitution}
\label{sect:B-series}

Consider any left pre-Lie algebra $(A,\rhd)$, and introduce a fictitious unit $\un$ such that $\un\rhd a=a\rhd\un=a$ for any $a\in A$. Due to the fact that $(T,\curvearrowright)$ is the free pre-Lie algebra with one generator \cite{ChaLiv}, there is for any $a\in A$ a unique pre-Lie algebra morphism $F_a:(T,\curvearrowright)\longrightarrow (A,\rhd)$ such that $F_a(\racine)=a$. A {\sl $B$-series\/} is an element of $hA[[h]]\oplus k.\un$ defined by:
\begin{equation}
\label{def:B-series}
	B(\alpha;a):=\alpha(\emptyset)\un+\sum_{s\in T}h^{v(s)}\frac{\alpha(s)}{\sigma(s)}F_a(s),
\end{equation}
where $\alpha$ is any linear form on $T\oplus k\emptyset$. It matches the usual notion of $B$-series \cite{HWL}, see Remark \ref{rmk:B-series}, when $A$ is the pre-Lie algebra of vector fields on $\R^n$, with $a\rhd b=\nabla_a b$ for the canonical flat torsion-free connection $\nabla$ on $\R^n$ (it is also convenient to set  $F_a(\emptyset)=\un$). The vector fields $F_a(t)$ for a tree $t$ are differentiable maps from $\R^n$ to $\R^n$ called {\sl{elementary differentials\/}}. $B$-series can be composed coefficientwise, as series in the indeterminate $h$ whose coefficients are maps from $\R^n$ to $\R^n$. The same definition with trees decorated by a set of colours $\Cal D$ leads to straightforward generalizations. For example $P$-series used in partitioned Runge-Kutta methods \cite{HWL} correspond to bi-coloured trees.\\

A slightly different way of defining $B$-series is the following: consider the unique pre-Lie algebra morphism $\Cal F_a:T\to hA[h]$ such that $\Cal F_a(\racine)=ha$. It respects the graduations given by the number of vertices and the powers of $h$ respectively, hence it extends to $\Cal F_a:\widehat T\to hA[[h]]$, where $\widehat T$ is the completion of $T$ with respect to the graduation. We further extend it to the empty tree by setting $\Cal F_a(\emptyset)=\un$. We have then:
\begin{equation}
\label{Bseriesbis}
	B(\alpha;a)=\Cal F_a\circ \wt\delta^{-1}(\alpha),
\end{equation}
where $\wt\delta$ is the isomorphism from $\widehat T\oplus k\emptyset$ to $(T\oplus k\emptyset)^*$ given by the normalized dual basis.\\

We restrict ourselves to $B$-series $B(\alpha;a)$ with $\alpha(\emptyset)=1$. Such $\alpha$'s are in one-to-one correspondence with characters of the algebra of forests (which is the underlying algebra of either $\Cal H_{\makebox{{\tiny{CK}}}}$ or $\wt{\Cal H}$) by setting:
\begin{equation*}
	\alpha(t_1\cdots t_k):=\alpha(t_1)\cdots\alpha(t_k).
\end{equation*}
The Hairer-Wanner theorem \cite[Theorem III.1.10]{HWL} says that composition of $B$-series corresponds to the convolution product of characters of $\Cal H_{\makebox{{\tiny{CK}}}}$, namely:
\begin{equation}
\label{hw-theorem}
	B(\beta;a)\circ B(\alpha;a)=B(\alpha * \beta,a),
\end{equation}
where linear forms $\alpha,\beta$ on $T\oplus k\emptyset$ and their character counterparts are identified modulo the above correspondence.\\

Let us now turn to substitution of $B$-series, following reference \cite{CHV}. The idea is to replace the vector field $a$ in a $B$-series $B(\beta;a)$ by another vector field $\wt a$ which expresses itself as a $B$-series, i.e. $\wt a=h^{-1}B(\alpha;a)$ where $\alpha$ is now a linear form on $T\oplus k\emptyset$ such that 
$\alpha(\emptyset)=0$. Such $\alpha$'s are again in one-to-one correspondence with characters of $\wt {\Cal H}$.

\begin{prop}
\label{substitution}
Let $\alpha,\beta$ be linear forms on $T$, we have:
\begin{equation*}
	B\big(\beta;\frac 1h B(\alpha;a)\big)=B(\alpha\star\beta;a),
\end{equation*}
where $\alpha$ is multiplicatively extended to
forests, $\beta$ is seen as an infinitesimal character of $\Cal H_{\makebox{{\tiny{CK}}}}$ and where $\star$ is the convolution product associated to the coproduct $\Delta_{\wt{\Cal H}}$ .
\end{prop}

\begin{proof}
Let $\alpha\in (T\oplus k\emptyset)^*$ such that $\alpha(\emptyset)=0$. By universal property of the free pre-Lie algebra there exists a unique pre-Lie algebra morphism $A_\alpha: T\to \widehat T$ such that $A_\alpha(\racine)=\wt\delta^{-1}(\alpha)$, which obviously extends to $A_\alpha: \widehat T\to \widehat T$. The following equality between pre-Lie algebra morphisms is a straightforward consequence of \eqref{Bseriesbis}:
\begin{equation*}
	\Cal F_{\frac 1hB(\alpha;a)}=\Cal F_a\circ A_\alpha.
\end{equation*}
Applying \eqref{Bseriesbis} twice we have then:
\begin{equation*}
	B\big(\beta;\frac 1h B(\alpha;a)\big)=\Cal F_a\circ A_\alpha\circ \wt\delta^{-1}(\beta),
\end{equation*}
whereas:
\begin{equation*}
	B\big(\alpha\star\beta;a)=\Cal F_a\circ\wt\delta^{-1}(\alpha\star\beta).
\end{equation*}
We are done if we can prove that:
\begin{equation*}
	A_\alpha(t)=\wt\delta^{-1}(\alpha\star\wt\delta_t)
\end{equation*}
for any $t\in\widehat T$. This is obviously true for $t=\racine$ according to the definition of $A_\alpha$ and \eqref{CHV2}, hence for any $t$ because $A_\alpha$ is a pre-Lie algebra morphism, and $\wt\delta^{-1}\big(\alpha\star\wt\delta(-)\big)$ as well according to the distributivity equation \eqref{CHV4}. This proves the proposition.
\end{proof}

\section{The CHV-Murua $\omega$ map and quasi-shuffle products}
\label{sect:Murua}

We are interested in the inverse $L$ of the character $E$ for the convolution product $\star$. We firstly give a quick proof of the
following result due to A. Murua \cite[Remark 11]{M}:

\begin{thm}
$$
	\omega:=L\star\delta_{\racine}=\log^*\delta.
$$
\end{thm}

\begin{proof}
We compute:
\begin{eqnarray*}
	E\star \log^*\delta&=&\log^*(E\star\delta)\\
					&=&\log^*\exp^*\delta_{\racine}\\
					&=&\delta_{\racine}.
\end{eqnarray*}
Hence $E\star \log^*\delta=E\star L\star\delta_{\racine}$, from which the theorem follows. 
\end{proof}

We call $\log^*\delta$ the CHV-Murua $\omega$-map after Chartier, Hairer, Vilmart \cite{CHV}, and Murua \cite{M}. One can find in those references some algorithms and recursive equations for computing the coefficients of $\omega$, which we will recover in the following subsection by means of a quasi-shuffle product. In fact, recalling the (free) pre-Lie structure on rooted trees corresponding to Connes--Kreimer's 
Hopf algebra, see \cite{ChaLiv} for details,  one deduces from the foregoing and \cite{EM, EMdendeq} that:
$$
	\Omega':=\sum_{n=0}^{\infty} \frac{B_n}{n!} L^{(n)}_{\curvearrowright}[\Omega'](\! \racine) = \sum_{t \in T} \frac{\omega(t)}{\sigma(t)}t.
$$ 
Here,  $L_{\curvearrowright}[a](b):=a \curvearrowright b$ is the left pre-Lie multiplication (in the free pre-Lie algebra over one generator $\begin{matrix}  \\[-0.5cm]  \racine \end{matrix}$). The first few terms on the left hand side are:
\begin{eqnarray*}
	\Omega'&=&\racine -\frac12 \racine \curvearrowright\! \racine 
							+ \frac14 (\racine \curvearrowright\! \racine) \curvearrowright\! \racine 
							+ \frac{1}{12} \racine \curvearrowright\! (\racine \curvearrowright\! \racine)  + \cdots\\
			&=&\racine - \frac12 \arbrea + \frac13 \arbreba + \frac{1}{12} \arbrebb +\cdots			
\end{eqnarray*}
A very interesting approach in terms of noncommutative
symmetric functions, along the lines of \cite{GKLLRT}, has been developed by W. Zhao \cite{Z}. Compare with \cite{C}\footnote{We would like to thank F. Chapoton who first pointed us to the link between the work \cite{EM} and the results in \cite{C}, and also pointed reference \cite{Z} to us.} and the second line of the table in subsection \ref{ssect:omegaval}.


\subsection{Quasi-shuffle products}
\label{ssect:quasi-shuffle}

\begin{defn}\label{def:qsh}
Let $k,l,r\in\N$ with $k+l-r>0$. A {\rm $(k,l)$-quasi-shuffle\/} of type $r$ is a surjective map $\pi$ from $\{1,\ldots ,k+l\}$ onto
$\{1,\ldots, k+l-r\}$ such that $\pi(1)<\cdots <\pi(k)$ and $\pi(k+1)<\cdots <\pi(k+l)$. We shall denote by $\mop{mix-sh}(k,l;r)$ the set of $(k,l)$-quasi-shuffles of type $r$. The elements of $\mop{mix-sh}(k,l;0)$ are the ordinary $(k,l)$-shuffles. Quasi-shuffles are also called {\rm mixable shuffles\/} or {\rm stuffles\/}. We denote by $\mop{mix-sh}(k,l)$ the set of $(k,l)$-quasi-shuffles (of any type).
\end{defn}

Let $A$ be a commutative (not necessarily unital) algebra. We denote by
$(a,b)\mapsto [ab]$ the product of $A$. Let $\Delta$
be the deconcatenation coproduct on ${\mathcal A}=T(A)=\bigoplus_{k\ge 0}A^{\otimes
  k}$, let $\diamond$ the product
on ${\mathcal A}$ defined by:
$$
	(v_1\cdots v_k)\diamond (v_{k+1}\cdots v_{k+l}) = \sum_{\pi\in\smop{mix-sh}(k,l)}w^\pi_1\cdots w^\pi_{k+l-r},$$
with~:
$$
	w^\pi_j=\prod_{i\in\{1,\ldots,k+l\}\atop \pi(i)=j}v_i
$$
(the product above is the product of $A$, and contains only one or two terms). The quasi-shuffle product $\diamond$ is commutative, and can also be recursively defined for $u=u_1\cdots u_k$ and $v=v_{1}\cdots v_{l}$ by:
\begin{equation}
\label{def:qs}
	u \diamond v = u_1\big( (u_2\cdots u_k)\diamond v\big) + v_1(u \diamond
    (v_2\cdots v_l)\big) + [u_1v_1]\big((u_2\cdots u_k)\diamond (v_2\cdots v_l)\big),
\end{equation}
or alternatively:
\begin{equation}
\label{def:qs2}
	u \diamond v = \big( (u_1\cdots u_{k-1})\diamond v\big)u_k + (u \diamond
    (v_1\cdots v_{l-1})\big)v_l + \big((u_1\cdots u_{k-1})\diamond (v_1\cdots
    v_{l-1})\big)[u_kv_l].
\end{equation}
Recall (see \cite{H1}) that $({\mathcal A} ,\diamond, \Delta)$ is a connected graded Hopf algebra, where $\Delta$ stands for the usual deconcatenation coproduct:
$$
	\Delta(u_1\cdots u_k)=\sum_{r=0}^s u_1\cdots u_r\otimes u_{r+1}\cdots u_k.
$$

\begin{rmk}{\rm{
When the multiplication of the algebra $A$ is set to $0$, the quasi-shuffle product $\diamond$ reduces to the ordinary shuffle product $\shu$.}}
\end{rmk}

Let $A$ be the field $k$ itself, and let ${\mathcal A}=T(A)=k[x]$ be the free
associative algebra over one generator $x$, which is identified with the unit
$1$ of the field $k$. We  equip ${\mathcal A}=T(k)$ with the
commutative quasi-shuffle product defined above. Note
that in the definition all the letters are equal to $x$ and we have
$[xx]=x$. The algebra $({\mathcal A},\diamond)$ is filtered, but not graded with respect to the number of
letters in a word. Example:
\begin{eqnarray*}
	x^2 \diamond x^2 &=& 2 x(x \diamond xx) + [xx] (x \diamond x)\\
	                         &=& 2 xxxx 
	                         		+ 2 xx (x \diamond x) 
						+ 2  x [xx] x 
						+ 2  [xx]xx 
	                         		+    [xx][xx] \\
				   &=& 2 xxxx 
	                         		+ 4 xxxx 
							+ 2  xx[xx] 
						+ 2 x [xx] x 
						+ 2 [xx]xx 
	                         		+  [xx][xx] \\
				   &=& 6 xxxx 
	                         		+ 2 x [xx] x 
						+ 2 [xx]xx 
						       + 2  xx[xx] 
	                         		+   [xx][xx]\\
&=& 6x^4+6x^3+x^2.
\end{eqnarray*}


\subsection{A Hopf algebra morphism}
\label{ssect:HopfAlgMorph}

Let $L:\Cal A \to \Cal A$ be the linear operator which multiplies with the single-lettered word $x$ on the right:
$$
	L(U):=Ux.
$$
It is immediate to show that $L$ is a coalgebra cocycle, namely:
$$
	\Delta\circ L=L\otimes \un+(\mop{Id}\otimes L)\circ\Delta
$$
where $\un:k\to {\mathcal A}$ is the unit. Due to the universal property of the Connes--Kreimer Hopf algebra, there is a unique Hopf algebra morphism $\Lambda:{\mathcal H}_{\makebox{{\tiny{CK}}}}\to {\mathcal A}$ such that:
\begin{equation}
\label{cocycles}
	\Lambda\circ B_+=L\circ\Lambda.
\end{equation}
In other words the morphism $\Lambda$ is recursively defined by $\Lambda(\!\begin{matrix}  \\[-0.5cm]  \racine \end{matrix})=x$ and:
\begin{equation}
\label{phi-rec}
	\Lambda(t) =L(\Lambda(t_1) \diamond \cdots \diamond \Lambda(t_n))= \bigl(\Lambda(t_1) \diamond \cdots \diamond \Lambda(t_n)\bigr)x
\end{equation}
for $t=B_+(t_1\cdots t_n)$. That is, briefly said, each rooted tree is mapped
to a particular polynomial by putting a quasi-shuffle product in each
branching point. The morphism $\Lambda$ is obviously surjective (because the
ladder with $n$ vertices $B_+^n(\un)$ is mapped on $x^n$). In order to study
the kernel of $\Lambda$ we introduce a quasi-shuffle like product on rooted trees
(which is not associative). For two rooted trees $a=B_+(a_1\cdots a_n)$ and
$b=B_+(b_1\cdots b_p)$ recall (\cite{CHV}, \cite{M}) the Butcher product $a\circ b$ which is obtained
by grafting the tree $b$ to the root of the tree $a$, namely:
\begin{equation*}
	a\circ b=B_+(a_1\cdots a_nb),
\end{equation*}
whereas the merging product $a\times b$ is obtained by merging the roots,
namely:
\begin{equation*}
	a\times b=B_+(a_1\cdots a_nb_1\cdots b_p).
\end{equation*}
The merging product is associative and commutative, whereas the Butcher product is
neither associative nor commutative. We remind the quasi-shuffle like product
on rooted trees defined by:
\begin{equation*}
	a\bowtie b:=a\circ b+b\circ a + a\times b,
\end{equation*}
which is implicitly used in \cite{CHV} (see Proposition 4.3 therein). 

\begin{prop}\label{qsqs}
For any rooted trees $a$ and $b$ we have:
\begin{equation*}
	\Lambda(a\bowtie b)=\Lambda(a)\diamond\Lambda(b).
\end{equation*}
\end{prop}

\begin{proof}
According to the recursive definition \eqref{def:qs2} we have:
 \allowdisplaybreaks{
\begin{eqnarray*}
\lefteqn{\Lambda(a)\diamond\Lambda(b)=\big(\Lambda(a_1\cdots a_n)x\big)\diamond\big(\Lambda(b_1\cdots b_p)x\big)}\\
&=&\Big(\big(\Lambda(a_1)\diamond\cdots
\diamond \Lambda(a_n)\big)x\Big)\diamond\Big(\big(\Lambda(b_1)\diamond\cdots
\diamond \Lambda(b_p)\big)x\Big)\\
&=&\Big(\big(\Lambda(a_1)\diamond\cdots
\diamond \Lambda(a_n)\big)x\diamond \Lambda(b_1)\diamond\cdots
\diamond \Lambda(b_p)\Big)x
+\Big(\big(\Lambda(a_1)\diamond\cdots
\diamond \Lambda(a_n)\diamond\big(\Lambda(b_1)\diamond\cdots
\diamond \Lambda(b_p)\big)x\Big)x\\
&&\hspace{1cm}+\Big(\Lambda(a_1)\diamond\cdots
\diamond \Lambda(a_n)\diamond\Lambda(b_1)\diamond\cdots
\diamond \Lambda(b_p)\Big)x\\
&=&\Big(\big(\Lambda(a_1\cdots a_n)x\big)\diamond \Lambda(b_1\cdots b_p)\Big)x
+\Big(\Lambda(a_1\cdots a_n)\diamond\big(\Lambda(b_1\cdots b_p)\big)x\Big)x
+\Big(\Lambda(a_1\cdots a_n)\diamond \Lambda(b_1\cdots b_p)\Big)x\\
&=&\Big(\Lambda(ab_1\cdots b_p)+\Lambda(a_1\cdots a_n b)+\Lambda(a_1\cdots a_nb_1\cdots
b_p)\Big)x\\
&=&\Lambda(a\bowtie b).
\end{eqnarray*}}
\end{proof}

\begin{thm}\label{kerphi}
$\mop{Ker}\Lambda$ is the smallest ideal of ${\mathcal H}_{\makebox{{\tiny{\rm{CK}}}}}$ stable under $B_+$ and containing the expressions $a \bowtie b-ab$ for any rooted trees $a$ and $b$.
\end{thm}

\begin{proof}
According to equation \eqref{cocycles} the kernel of $\Lambda$ is clearly stable under $B_+$, and moreover $u\in\mop{Ker}\Lambda$ if and only if $B_+(u)\in\mop{Ker}\Lambda$. According to Proposition \ref{qsqs}, any expression $a\bowtie b-ab$ belongs to $\mop{Ker}\Lambda$. Now let $u\in\mop{Ker}\Lambda$. Let
$|u|$ be the filtration degree of $u$, i.e.
$$
	|u|=\mop{min }\big\{n,\,u\in\bigoplus_{k=0}^n{\mathcal H}_{\makebox{{\tiny{CK}}}}^{(n)}\big\}.
$$
The only non-trivial element in $\mop{Ker}\Lambda$ of filtration degree $\le 2$
writes $u=\bullet\bowtie\bullet-\bullet\bullet$. Let us prove Theorem
\ref{kerphi} by induction on the filtration degree: suppose that any
$u\in\mop{Ker}\Lambda$ of filtration degree $\le n$ is of the type described in
the theorem, and let $u\in\mop{Ker}\Lambda$ of filtration degree $n+1$. Now
proceed by induction on the maximal number of subtrees, namely:
$$
	d(u)=\mop{min }\big\{p,\,u\in\bigoplus_{k=0}^p\Cal S^k(T)\big\},
$$
where $T$ is vector space freely spanned by the rooted trees. If
$d(u)=1$, then $u$ is a linear combination of trees, hence $u=B_+(v)$ with
$|v|=|u|-1=n$. As $v\in\mop{Ker}\Lambda$ the theorem is proved in this case. If
$d(u)=p>1$ then $u$ is a linear combination of forests. Suppose that the
theorem is verified up to $d(u)=p$, and let $u$ such that $|u|=n+1$ and
$d(u)=p+1$. Write the unique decomposition $u=v+w$ where $d(w)=p$ and $v\in
\mathcal{S}^{p+1}(T)$. Define $\wt v$ by replacing any forest $v_1\cdots
v_{p+1}$ by $(v_1\bowtie v_2)\cdots v_{p+1}$ in the explicit expression of
$v$. It is clear that $d(\wt v)=d(u)-1=p$ and that $v-\wt v$ belongs to the
ideal generated by $\{a\bowtie b-ab,\, a,b\in{T}\}$. The theorem
follows since we have $u=v-\wt v+(\wt v+w)$, with $\wt v+w\in\mop{Ker}\Lambda$
and $d(\wt v+w)=d(u)-1=p$, as well as $|\wt v+w|\le n+1$.
\end{proof}


\subsection{Applications to the CHV-Murua $\omega$ map}

\begin{lem}
Let $\widetilde\delta:\Cal A\to k$ the linear map defined by $\widetilde\delta(\un)=\widetilde\delta(x)=1$, and $\widetilde\delta(u)=0$ for any word of length $s\ge 2$. Then $\widetilde\delta$ is a character of $\Cal A$ for the quasi-shuffle
product $\diamond$, and we have:
\begin{equation*}
	\delta=\widetilde\delta\circ\Lambda.
\end{equation*}
\end{lem}

\begin{proof}
If $u$ and $v$ are words of respective length $l(u)$ and $l(v)$, then
$u\diamond v$ is a linear combination of words of
length bigger than $\mop{sup}\big(l(u),l(v)\big)$. Hence $\widetilde\delta(u\diamond v)=0$ unless $u$ and $v$ are single-lettered
words. We have then in this case:
\begin{eqnarray*}
\widetilde\delta(u\diamond v)&=&\widetilde\delta(uv+vu+[uv])\\
						   &=&\widetilde\delta([uv])\\
						   &=&\widetilde\delta(u)\widetilde\delta(v),
\end{eqnarray*}
which proves that $\widetilde\delta$ is a character. The second assertion comes from a straightforward computation.
\end{proof}

Now set:
$$
	\widetilde\omega:=\log^*\widetilde\delta=K-\frac{K^{*2}}{2}+\frac{K^{*3}}{3}+\cdots,
$$
with $K:=\widetilde\delta-\varepsilon$. As $K$ vanishes on words containing more than one letter, we have for any word $u=u_1\cdots u_s$:
\begin{equation*}
	\widetilde\omega(u)=(-1)^{s+1}\frac{K^{*s}}{s}(u)=\frac{(-1)^{s+1}}{s}.
\end{equation*}
Due to the fact that $\Lambda$ is a Hopf algebra morphism we obviously get:

\begin{thm}\label{omegatilde}
\begin{equation*}
	\omega=\widetilde\omega\circ\Lambda.
\end{equation*}
\end{thm}

\begin{cor}\label{formule}
Let $\omega_s(t)$ the number of words of length $s$ in $\Lambda(t)$, i.e:
\begin{equation*}
	\Lambda(t)=\sum_{s=1}^{|t|}\omega_s(t)x^s.
\end{equation*}
Then we have:
\begin{equation*}
	\omega(t)=\sum_{s\ge 1}\frac{(-1)^{s+1}}{s}\omega_s(t).
\end{equation*}
\end{cor}

We can easily see that $\omega_s(t)$ is the number of terms in the iterated coproduct:
$$
	\Delta^{(s)}(t)=\sum_{(t)}t^{(1)}\otimes\cdots\otimes t^{(s)}
$$
such that each $t_j$ is a bullet-forest $\bullet^{k_j}$ (with obviously $k_1+\cdots +k_s=|t|$). In view of formula \eqref{iter-coprod} it equals the number of ordered partitions of the tree $t$ into $s$ bullet-subforests. Hence our $\omega_s(t)$ coincides with Murua's one, and Corollary \ref{formule} gives exactly identity (40) in \cite{M}.

\begin{prop}
For any rooted trees $a$ and $b$, we have $\omega(a\bowtie b)=0$.
\end{prop}

\begin{proof}
As a consequence of Proposition \ref{qsqs} and Theorem  \ref{omegatilde}, we have $\omega(a\bowtie b)=\wt\omega \big(\Lambda(a)\diamond\Lambda(b)\big)=0$, because $\wt\omega$ is an infinitesimal character of ${\mathcal A}$, and as such vanishes on any
non-trivial diamond product.
\end{proof}

\begin{rmk}{\rm{
 Propositions 4.3, 4.5 and 4.7 in \cite{CHV} can be retrieved from this result.}}
\end{rmk}


\subsection{Generalized multinomial coefficients}

A simple formula is given for the quasi-shuffle product on $\Cal A$:
$$
	x^k \diamond x^l =
	\sum_{k=0}^{n} \binom{l+k-r}{k}\binom{k}{r} x^{l+k-r}.
$$
Indeed, the number of $(k,l)$-quasi-shuffles of type $r$ is equal to:
\begin{equation*}
	\mop{qsh}(k,l;r)={k+l-r\choose k}{k\choose r}=\frac{(k+l-r)!}{(k-r)!(l-r)!r!}.
\end{equation*}
This can be seen as follows: there is ${k+l-r\choose k}$ choices for the images $\pi_1,\dots, \pi_k$ inside $\{1,\ldots k+l-r\}$. The whole quasi-shuffle $\pi$ is then determined by the overlaps, i.e. the choice of a subset $E$ of $\{1,\ldots l\}$ containing $r$ elements, such that $j\in E$ if and only if $\pi^{-1}\big(\pi(j)\big)$ has exactly two elements. There are ${k\choose r}$ choices of overlaps, which proves the claim. We can also define the quasi-shuffle multinomial coefficient $\mop{qsh}(k_1,\ldots k_n;r)$ as the number of surjective maps $\pi:\{1,\ldots, \sum k_j\}\to \{1,\ldots, \sum k_j-r\}$ such that $\pi_{k_1+\cdots +k_j+1}<\cdots <\pi_{k_1+\cdots +k_{j+1}}$ for any $j=0,\ldots ,n-1$. It is also given by the coefficient of $x^{k_1+\cdots +k_n-r}$ in $x^{k_1}\diamond\cdots\diamond x^{k_n}$. Of course when $r=0$ we recover the usual mutinomial coefficients.

\bigskip

The coefficients $\omega_s(t)$ defined by Murua (see preceding section) can be interpreted as tree versions of these quasi-shuffle multinomial coefficients. To see this let us change the notations of the preceding section, by setting for any rooted tree $t$:
\begin{equation*}
	C_s(t):=\omega_{|t|-s}(t).
\end{equation*}
In other words, $C_s(t)$ is the coefficient of $x^{|t|-s}$ in the polynomial $\Lambda(t)$.

\begin{prop}\label{qsbc}
For any rooted tree $t=B_+(t_1,\ldots ,t_n)$ we have:
\begin{equation*}
	C_s(t)=\sum_{j=0}^s
	          \sum_{r_1+\cdots r_n=s-j, \atop r_1,\ldots, r_n\ge 0}\mop{qsh}(|t_1|-r_1,\ldots,|t_n|-r_n;j)\,C_{r_1}(t_1)\cdots C_{r_n}(t_n).
\end{equation*}
\end{prop}

\begin{proof} This amounts to the equality:
$$
	\Lambda(t)=\big(\Lambda(t_1)\diamond\cdots\diamond\Lambda(t_n)\big)x
$$
by considering the coefficient of $x^{|t|-s}$ in both sides.
\end{proof}

Applying Proposition \ref{qsbc} to the generalized corolla
${\mathcal C}_{k_1,\ldots, k_n}=B_+(E_{k_1},\ldots E_{k_n})$ (here $E_j$ stands
for the ladder with $j$ vertices $B_+^j(\bullet)$) we find:
\begin{equation*}
	C_s({\mathcal C}_{k_1,\ldots, k_n})=\mop{qsh}(k_1,\ldots k_n;s).
\end{equation*}
Proposition \ref{qsbc} for $s=0$ reduces to:
\begin{equation*}
	C_0(t)=\frac{(|t_1|+\cdots +|t_n|)!}{|t_1|!\cdots |t_n|!}C_0(t_1)\cdots C_0(t_n),
\end{equation*}
leading to:
\begin{equation*}
	C_0(t)=\frac{|t|!}{t!}.
\end{equation*}
These are generalizations of the usual multinomial coefficients, which can be
recovered as:
\begin{equation*}
	\frac{(k_1+\cdots + k_n)!}{k_1!\cdots k_n!}=C_0({\mathcal C}_{k_1,\ldots, k_n}).
\end{equation*}


\section{Some computations}
\label{sect:comp}


\subsection{Antipode computation for some trees} We list the values of the antipode $S_\sigma : \Cal H \to \Cal H$ up to order 5. Observe the conservation of the number of edges.

 \allowdisplaybreaks{	
\begin{eqnarray*}
S_\sigma(\racine)&=&\racine\\
&&\\
S_\sigma(\arbrea)&=&-\arbrea\\
&&\\
S_\sigma(\arbreba)&=&-\arbreba+2\arbrea\arbrea\\
S_\sigma(\arbrebb)&=&-\arbrebb+\arbrea\arbrea\\
&&\\
S_\sigma(\arbreca)&=&-\arbreca+5\arbreba\arbrea-5\arbrea\arbrea\arbrea\\
S_\sigma(\arbrecb)&=&-\arbrecb+4\arbreba\arbrea+\arbrebb\arbrea-6\arbrea\arbrea\arbrea\\
S_\sigma(\arbrecc)&=&-\arbrecc+2\arbreba\arbrea+3\arbrebb\arbrea-5\arbrea\arbrea\arbrea\\
S_\sigma(\arbrecd)&=&-\arbrecd+2\arbrebb\arbrea-\arbrea\arbrea\arbrea\\
&&\\
S_\sigma(\arbreda)&=&-\arbreda+6\arbreca\arbrea+3\arbreba\arbreba-21\arbreba\arbrea\arbrea+14\arbrea\arbrea\arbrea\arbrea\\
S_\sigma(\arbredb)&=&-\arbredb+2\arbreca\arbrea+2\arbrecb\arbrea+\arbreba\arbrebb
-18\arbreba\arbrea\arbrea-\arbrebb\arbrea\arbrea+18\arbrea\arbrea\arbrea\arbrea\\
S_\sigma(\arbredh)&=&-\arbredh+2\arbrecd\arbrea+\arbrebb\arbrebb-3\arbrebb\arbrea\arbrea+\arbrea\arbrea\arbrea\arbrea\\
&&\\
S_\sigma(\arbreea)&=&-\arbreea+7\arbreda\arbrea+7\arbreca\arbreba
-28\arbreca\arbrea\arbrea-28\arbreba\arbreba\arbrea+84\arbreba\arbrea\arbrea\arbrea-42\arbrea\arbrea\arbrea\arbrea\arbrea\\
S_\sigma(\arbreez)&=&-\arbreez+2\arbredh\arbrea+2\arbrecd\arbrebb
-3\arbrecd\arbrea\arbrea-3\arbrebb\arbrebb\arbrea
+4\arbrebb\arbrea\arbrea\arbrea-\arbrea\arbrea\arbrea\arbrea\arbrea.\\
\end{eqnarray*}}


\subsection{Values of $\omega$ up to degree 5}
\label{ssect:omegaval}
Here are the values of $L$ and $L_\sigma$ (i.e. $\omega$ and $\omega_\sigma$) on the trees up to 5 vertices (compare with \cite{CHV,M} and \cite{C}):

\begin{equation*}
\begin{disarray}{c||c|c|cc|cccc|ccccccccc}
t&\racine &\arbrea &\arbreba &\arbrebb &\arbreca &\arbrecb &\arbrecc &\arbrecd &\arbreda
&\arbredb &\arbredc &\arbredd &\arbrede &\arbredf &\arbredz &\arbredg &\arbredh\\ 
&&&&&&&&&&&&&&&&&\\
\hline
&&&&&&&&&&&&&&&&&\\
\omega(t)&1&-\frac 12 &\frac 13&\frac 16&-\frac 14&-\frac 16&-\frac 1{12}&0
&\frac 15&\frac 3{20}&\frac 1{10}&\frac 1{30}&\frac 1{20} &\frac 1{30}&\frac 1{60} &-\frac 1{60} &-\frac 1{30}\\
&&&&&&&&&&&&&&&&&\\
\hline
&&&&&&&&&&&&&&&&&\\
\frac{\omega(t)}{\sigma(t)}&1&-\frac 12 &\frac 13&\frac 1{12}&-\frac 14&-\frac 1{12}&-\frac 1{12}&0
&\frac 15&\frac 3{40}&\frac 1{10}&\frac 1{180}&\frac 1{20} &\frac 1{60}&\frac 1{120} &-\frac 1{120} &-\frac 1{720}
\end{disarray}
\end{equation*}

\bigskip 

\subsection*{Acknowledgements}
The authors would like to thank Fr\'ed\'eric Chapoton, Philippe Chartier, Ander Murua and Gilles Vilmart for enlightening discussions. As well we thank Lo\"\i c Foissy and Jose Gracia--Bond\`\i a for useful suggestions and comments. We also thank the referee for very helpful remarks and suggestions. Finally we thank Emmanuel Vieillard-Baron and Fr\'ed\'eric Fauvet for having pointed out an error in the antipode formula of section \ref{sect:antipo} in a previous version. The research of D.C. is fully supported by the European Union thanks to a Marie Curie Intra-European Fellowship (contract  number MEIF-CT-2007-042212).

\end{document}